\documentclass[12pt]{article}

\usepackage{amsmath,amssymb,mathtools,amsthm}
\mathtoolsset{showonlyrefs}
\usepackage{dsfont}
\usepackage{enumerate}
\usepackage{authblk}
\usepackage{subcaption}
\usepackage{float}
\usepackage{pgfplots}
\usepackage{tikz}
\usetikzlibrary{arrows.meta}
\usepackage[utf8]{inputenc}
\usepackage{xcolor}
\definecolor{refereered}{RGB}{155,25,45}
\newif\ifrefereecomments
\refereecommentstrue

\usepackage[
    inner=2.3cm,
    outer=2.3cm,
    top=2.7cm,
    marginparwidth=2.5cm,
    marginparsep=0.1cm
    ]
    {geometry}
\usepackage[pdffitwindow=false,
            plainpages=false,
            pdfpagelabels=true,
            pdfpagemode=UseOutlines,
            pdfpagelayout=SinglePage,
            bookmarks=false,
            colorlinks=true,
            hyperfootnotes=false,
            linkcolor=blue,
            urlcolor=blue!30!black,
            citecolor=green!50!black]{hyperref}
    
\newcommand{\R}{\mathds{R}}

\newcommand{\E}{\mathds{E}}

\newcommand{\norm}[1]{\left\lVert#1\right\rVert}
\newcommand{\abs}[1]{\left|#1\right|}
\newcommand{\1}{\mathds{1}}
\renewcommand{\P}{\mathds{P}}

\newcommand{\Id}{\textup{Id}}
\newcommand{\dd}{\mathrm{d}}

\newtheorem{theorem}{Theorem}[section]
\newtheorem{proposition}[theorem]{Proposition}
\newtheorem{lemma}[theorem]{Lemma}

\newtheorem{definition}[theorem]{Definition}

\newtheorem{maintheorem}{Theorem}

\newtheorem{assumption}{Assumption}
\newtheorem{assumptionalt}{Assumption}[assumption]
\newenvironment{assumptionp}[1]{
  \renewcommand\theassumptionalt{#1}
  \assumptionalt
}{\endassumptionalt}

\theoremstyle{definition}

\newtheorem{remark}[theorem]{Remark}

\newtheorem{example}[theorem]{Example}
\newtheorem*{heuristic}{Heuristic}

\newenvironment{nalign}[1][]{%
    \begin{equation}%
    \if\relax\detokenize{#1}\relax\else\label{#1}\fi
    \begin{aligned}[c]
}{%
    \end{aligned}%
    \end{equation}%
    \ignorespacesafterend
}

\title{Characterizing stochastic inertia: when does noise increase escape times?}
\date{}

\author[1,2]{Robin Chemnitz}
\author[1,3]{Maximilian Engel}
\author[4]{P\'eter Koltai}

\affil[1]{Freie Universität Berlin}
\affil[2]{Tulane University}
\affil[3]{University of Amsterdam}
\affil[4]{Universität Bayreuth}

\begin{document}

\maketitle

\begin{abstract}
    Motivated by observations in atmospheric modelling, we investigate the effect of ``stochastic inertia'': in certain dynamical systems, small noise induces trajectories to stay longer in a dynamically critical subset of state space than in the purely deterministic situation.
    Under mild general assumptions on the stochastic differential equation, we derive an integral formula along deterministic trajectories whose sign provides a sufficient condition for the occurrence of stochastic inertia. Excluding situations with open sets of infinite escape times and assuming uniformly elliptic diffusion, we additionally identify this integral formula as the zero noise derivative of the average escape time.
    The proofs make crucial use of Dynkin's formula and suitable adaptations of Freidlin--Wentzell theory.
We present an algorithmic approach to compute the integral formula and, by examples of saddle points and relaxation oscillators, demonstrate the following heuristic:
stochastic inertia appears in the presence of an unstable direction that predominantly leads the escape dynamics.
\end{abstract}

\tableofcontents

\section{Introduction}

The response of systems to varying noise strengths is a widely investigated phenomenon. A prominent case of this is studying how the system exits from or transitions between regions of interest in its state space. An early and important quantitative result in this direction is the Eyring--Kramers transition rate formula~\cite{eyring1935activated,kramers1940brownian}, improving on Arrhenius' law. It states that, asymptotically in the limit of vanishing noise, expected transition times of an overdamped diffusion between two potential wells' minima increase exponentially with the inverse of the noise strength. Additionally, it quantifies the dependence of these transition times on the geometry of the potential landscape in the vicinity of specific transition paths (so-called instantons) through the Hessian of the potential---see~\cite{bovier2004metastability,helffer2004quantitative} and the survey~\cite{berglund2013kramers}.
For more general, irreversible diffusions, an Eyring--Kramers-type transition rate formula has been established in~\cite{bouchet2016generalisation} using the Freidlin--Wentzell large-deviations theory~\cite{freidlin1998random}. For further descriptions of the behavior of irreversible diffusions, see, for instance~\cite{graham1989macroscopic,berglund2004noise,berglund2014noise}.

The above studies consider transitions between regions organized around central structures of the noise-free dynamics: fixed points and periodic orbits. While this is natural, one could also be interested in the behavior of the system in a region not determined by its dynamical properties. An example could be a certain atmospheric constellation with regard to its effect on agriculture. The question we ask will be, given a system starting in a point within a prescribed region, how can the addition of noise to the dynamics change the (expected) time to exit this region. 

Indeed, our motivation to consider this problem rests on a study by Dorrington and Palmer \cite{dorrington2023interaction}, who investigated the increase in regime lifetimes under additive white noise in a low-order paradigm model of a mid-latitude atmospheric flow. This phenomenon has been numerically observed earlier in low-order models of atmospheric flows by Kwasniok~\cite{kwasniok2014enhanced}. The first numerical and theoretical studies addressing this effect in one-dimensional discrete-time dynamical systems seem to have appeared in~\cite{franaszek1991influence} and~\cite{reimann1994suppression}, respectively. For similar effects see also~\cite{altmann2010noise} and~\cite[sections~4.1.2, 4.2.3, and 4.3.2]{lai2011transient}. A data-based computational approach to test the phenomenon from time-sampled trajectories of the deterministic system has appeared in~\cite{schoeller2026noise}.
In forced (thus non-stationary) systems, such a ``noise enhanced stability'' has been observed, for instance, in~\cite{dayan1992stochastic,mantegna1996noise}.

An oversimplified and informal summary of our results is as follows. If $E(x)$ denotes the deterministic escape time of a flow $\phi^t$ starting in $x$ from a prescribed region, then at leading order, the small-noise change in the expected exit time from this region for the system subjected to isotropic white noise is governed by the quantity
\[
\int_0^{E(x)} \Delta E( \phi^t x)\,\dd t.
\]
We shall now formalize our setting and the statement.

\subsection{Setting}
Consider the $d$-dimensional SDE
\begin{equation}\label{eq:SDE}
    \dd X^t_\sigma = v(X^t_\sigma)\, \dd t +  \sqrt{2\sigma}\, G(X_\sigma^t)\, \dd W^t,
\end{equation}
where $v:\R^d \to \R^d$ is a vector field, $G:\R^d \to \R^{d\times d}$ is a diffusion matrix, $W^t$ is a $d$-dimensional Brownian motion and $\sigma \geq 0$ governs the strength of the noise. 
We assume that $v$ is locally $C^2$ and that $G$ is locally Lipschitz continuous. Our main result, Theorem \ref{thm:stoch_inertia}, is formulated for general diffusion matrices $G$. In Theorem \ref{thm:derivative}, we impose further restrictions on $G$. Throughout, we only work on bounded subsets of $\R^d$ such that we may assume w.l.o.g.~that $v$ is globally $C^2$ and $G$ is globally Lipschitz continuous.

Let $M\subset \R^d$ be a bounded open domain with piecewise $C^2$ boundary $\partial M \subset \R^d$. We assume that $\partial M$ coincides with the boundary of the closure $\overline{M}$. We consider $X^0_\sigma \in M$ and study the escape time from the set $M$. Given a noise strength $\sigma \geq 0$, define the random escape time from $M$,
\begin{equation*}
    \tau_{\sigma}^M = \inf \{t >0 \mid X^t_\sigma \notin M\}.
\end{equation*}
We study the expectation of $\tau_{\sigma}^M$ in dependence on~$\sigma$ for each initial condition~$X_\sigma^0\in M$.  Given $x\in M$, define
\begin{equation*}
    E^M_\sigma(x) := \E_x[ \tau_{\sigma}^M],
\end{equation*}
where $\E_x$ is the expectation assuming $X^0_\sigma = x$. Without additional assumptions $E_\sigma^M(x)$ may be infinite. We study the function $\sigma \mapsto E_\sigma^M(x)$. 

For $\sigma = 0$, the deterministic escape time is given by $E_0^M(x)$ and may be infinite, e.g.~when $x$ is a fixed point of the vector field $v$. For uniformly elliptic noise, it follows from standard theory that $E^M_\sigma(x)$ is finite for all $\sigma>0$ since $M$ is assumed to be bounded. Furthermore, as $\sigma\to \infty$, the expected escape time $E^M_\sigma(x)$ decreases to $0$ for all initial conditions~$x\in M$. In \cite{kwasniok2014enhanced} and \cite{dorrington2023interaction}, the authors observed that for certain hyperbolic systems $E^M_\sigma(x)$ is not simply monotonically decreasing in~$\sigma$. Instead, for $\sigma$ close to $0$, they observed $E^M_\sigma(x) > E_0^M(x)$. We call this phenomenon \emph{stochastic inertia}. Note that this is only possible if $E^M_0(x) < \infty$.

\begin{definition}
    Let $x\in M$ for which $E_0^M(x) < \infty$. If there is $\sigma^* > 0$ such that $E_\sigma^M(x)> E_0^M(x)$ for all $\sigma \in (0, \sigma^*)$, we say there is \emph{stochastic inertia} in $x$.
\end{definition}

\subsection{Main results}

Our first main result is a sufficient condition for stochastic inertia in a point $x\in M$. To state the result, we introduce some notation for the deterministic trajectory, i.e.~for $X_0^t$. Let $\phi^t:\R^d \to \R^d$ be the solution map of the deterministic vector field $v$, i.e.~$\phi^0x = x$ and
\begin{equation}
    \partial_t \phi^tx = v(\phi^tx).
\end{equation}
By definition, for all $x\in M$ we find $\phi^tx\in M$ for $t\in[0,E_0^M(x))$. Assuming that $E_0^M(x)<\infty$, we find $\phi^{E_0^M(x)}x \in \partial M$. In that case, we denote the deterministic exit point of $x\in M$ by $\theta(x) := \phi^{E_0^M(x)}x$. Assuming that the boundary $\partial M$ is differentiable around $\theta(x)$, we say that $x$ leaves $M$ \emph{non-tangentially} if $v(\theta(x))$ is not tangential to $\partial M$ in the point $\theta(x)$. Note that $v(\theta(x)) = 0$ is not possible, since otherwise the trajectory of $x$ would never have reached $\theta(x)$ in finite time, as we assumed $v$ to be locally $C^2$. Hence, if $x$ leaves $M$ non-tangentially, there are $\delta > 0$ and $\epsilon>0$ such that $\text{dist}(\phi^{E_0^M(x) + \delta}x, M) > \epsilon$. Whenever the deterministic escape time $E_0^M$ is twice differentiable, we define 
\begin{equation}
    \Gamma(x) \coloneqq \operatorname{trace} \left( G(x)^\top \operatorname{Hess} \left[E_0^M(x) \right]  G(x)\right).
\end{equation}
In Section \ref{sec:computing}, we present an algorithmic method to compute $\Gamma(x)$ numerically. Note that in the case of isotropic additive noise, i.e. $G=\Id$, this term reduces to $\Gamma(x) = \Delta E_0^M(x)$.

The following is our first main result. Its proof is contained in Section \ref{sec:proofs}.
\begin{maintheorem}\label{thm:stoch_inertia}
    Let $x\in M$ with $E_0^M(x) < \infty$ such that the trajectory starting in~$x$ leaves $M$ at $\theta(x) \in \partial M$ non-tangentially. Assume that $\partial M$ is twice differentiable in a neighbourhood of~$\theta(x)$. Then, $E_0^M(x)$ is twice differentiable along the deterministic trajectory $\{\phi^t x\}_{t\in [0, E_0^M(x))}$. 
    If
    \begin{equation} \label{eq:stoch_inertia}
         \int_0^{E_0^M(x)} \Gamma(\phi^t x) \, \dd t > 0,
    \end{equation}
    then there is stochastic inertia in $x$. 
\end{maintheorem}

To obtain a more quantitative version of Theorem \ref{thm:stoch_inertia}, consider the derivative
\begin{equation}\label{eq:derivative_sigma}
    \mathcal D(x) = \partial_\sigma E^M_\sigma(x) \big|_{\sigma = 0}.
\end{equation}
Without further assumptions, it is, in general, not true that $\sigma \mapsto E_\sigma^M(x)$ is differentiable, or even continuous, in~$\sigma = 0$. See Example \ref{ex:discontinuity} for a case in which $E_0^M(x) < \infty$ but $\lim_{\sigma \to 0} E_\sigma^M(x) = \infty$. 
If the derivative $\mathcal D(x)$ is well-defined and positive, then for small $\sigma>0$, we find $E^M_\sigma(x)>E_0^M(x)$, and there is stochastic inertia in~$x$. Conversely, if $\mathcal D(x) < 0$, there cannot be stochastic inertia in~$x$. %

We impose the following two assumptions, one on the deterministic and one on the stochastic component, such that the derivative $\mathcal{D}(x)$ is well-defined and can be expressed in terms of an explicit formula.

\begin{assumptionp}{I}\label{ass:no_stable}
   The set of points $x\in M$ with $E_0^M(x) = \infty$ has empty interior.
\end{assumptionp}
\begin{assumptionp}{II}\label{ass:uniformly elliptic}
The noise is \emph{uniformly elliptic}, i.e., there is a constant $\Lambda>0$ such that for all $x\in \R^d$ and all $\xi \in \R^d$
    \begin{equation}
        \Lambda^{-1}|\xi|^2 \leq |G(x)\xi|^2  \leq \Lambda |\xi|^2.
    \end{equation}
\end{assumptionp}
The two assumptions together play the role of preventing the exit times from taking extremely large values away from a small neighborhood of the deterministic trajectory associated with the point in question.

\begin{maintheorem}\label{thm:derivative}
    Let $x\in M$ be a point that satisfies the conditions of Theorem \ref{thm:stoch_inertia} and let Assumption \ref{ass:no_stable} and Assumption \ref{ass:uniformly elliptic} be satisfied.
    
    Then, $\mathcal{D}(x)$ is well-defined and given by
    \begin{equation}\label{eq:thm}
        \mathcal D(x) = \int_0^{E_0^M(x)} \Gamma(\phi^tx) \, \dd t.
    \end{equation}
\end{maintheorem}
In particular, while Theorem~\ref{thm:stoch_inertia} only indicates that positivity of the right-hand side of~\eqref{eq:thm} implies stochastic inertia, Theorem~\ref{thm:derivative} extends this by asserting that, under the additional assumptions, negativity of this integral indicates the \emph{absence of} stochastic inertia. 

\begin{remark}\label{rmk:scaling}
    Note that in the SDE \eqref{eq:SDE} the noise strength is given by $\sqrt{2 \sigma}$, and $\mathcal{D}(x)$ is the derivative of the expected escape time with respect to $\sigma$. Theorem \ref{thm:derivative} shows that $\mathcal{D}(x)$ can take non-zero, finite values. This shows that taking the derivative with respect to the square of the noise strength is the correct scaling.
\end{remark}

We can quickly obtain an intuition of how the central expression~\eqref{eq:thm} in the theorems arises. The
generator of~\eqref{eq:SDE} is
$L_\sigma = v\cdot\nabla + \sigma\operatorname{trace}\!\big(G^\top(\operatorname{Hess}\cdot)G\big)$,
and $E^M_\sigma$ formally solves the Dirichlet problem $L_\sigma E^M_\sigma=-1$
in $M$ with $E^M_\sigma=0$ on $\partial M$; see~\cite{schuss2009theory}. Inserting a regular expansion
$E^M_\sigma = u_0 + \sigma u_1 + \mathcal{O}(\sigma^2)$, each $u_k$ vanishing on
$\partial M$, and collecting powers of $\sigma$ gives
\[
  v\cdot\nabla u_0 = -1,
  \qquad
  v\cdot\nabla u_1 = -\operatorname{trace}\!\big(G^\top[\operatorname{Hess}u_0]G\big).
\]
The first equation identifies $u_0 = E_0^M$, so that the right-hand side of the
second is $-\Gamma$. The second is then a transport equation along the
deterministic flow: $\tfrac{\mathrm{d}}{\mathrm{d}t}u_1(\phi^t x) = -\Gamma(\phi^t x)$,
and integrating from $x$ to its exit point $\theta(x)$, where $u_1$ vanishes,
yields
\[
  u_1(x) = \int_0^{E_0^M(x)} \Gamma(\phi^t x)\,\mathrm{d}t .
\]
This computation is classical. In $\R^2$ and for $G=\mathrm{Id}$, Levinson
\cite{levinson1950first} made the expansion rigorous for
$\varepsilon\Delta u + v\cdot \nabla u + Cu = D$, under hypotheses that are,
however, global and considerably stronger than the ones we impose: all data of
class $C^6$ and the vector field needs to be free of equilibria and of
periodic orbits there. (For the exit time problem one has $C\equiv 0$, so the
alternative hypothesis $C<0$ is not available.) 
We also note that~\cite{matkowsky1977exit,naeh1990direct} consider the asymptotics of the expected exit time for small noise in the case of~$E_0^M \equiv \infty$.

In \cite{dorrington2023interaction} the authors give a heuristic explanation for stochastic inertia in a one-dimensional setting where the deterministic dynamics $\phi^t$ is ``convex'', in the sense that trajectories accelerate towards the boundary. In our setting, such a convexity corresponds to $\Gamma(\phi^{t} x)> 0$, which, by Theorem \ref{thm:stoch_inertia}, implies stochastic inertia. Therefore, Theorem \ref{thm:stoch_inertia} can be understood as a mathematically rigorous version of the reasoning of \cite{dorrington2023interaction}. In Section \ref{sec:fixed_point} we study stochastic inertia around fixed points with the help of Theorem \ref{thm:stoch_inertia}. The quintessence of that section is that we expect stochastic inertia around fixed points that have a single unstable direction significant for the escape from~$M$.

\subsection{Example}

The main idea behind the phenomenon and its connection to the formula given in the theorems above can be illustrated by a one-dimensional process, where~$v$ is positive, strictly monotonically increasing and such that the flow $\phi^t$ exists globally in time and space. Fix some point $x>0$ and consider the map $\Phi: t \mapsto \phi^t x$. Note that $\Phi$ is strictly convex, because $v$ is strictly increasing. Set $M = (-x^*, x^*)$ for some $x^* > x$, denote $T:= E^M_0(x)$ and note that $x^* = \theta(x)$, cf.~Fig.~\ref{fig:cartoon}. Now, since $E^M_0(\phi^tx) = T-t$, we have for any point $y\in [x,x^*]$ that $E^M_0(y) = T - \Phi^{-1}(y)$. Since $\Phi:[0,T] \to [x,x^*]$ is bijective, smooth, and strictly convex, we have that $\Phi^{-1}$ is strictly concave and hence $E^M_0: y \mapsto T-\Phi^{-1}(y)$ is again strictly convex, implying that $\smash{ \int_0^{E^M_0(x)} (E^M_0)''(\phi^t x)\,\dd t > 0 }$.
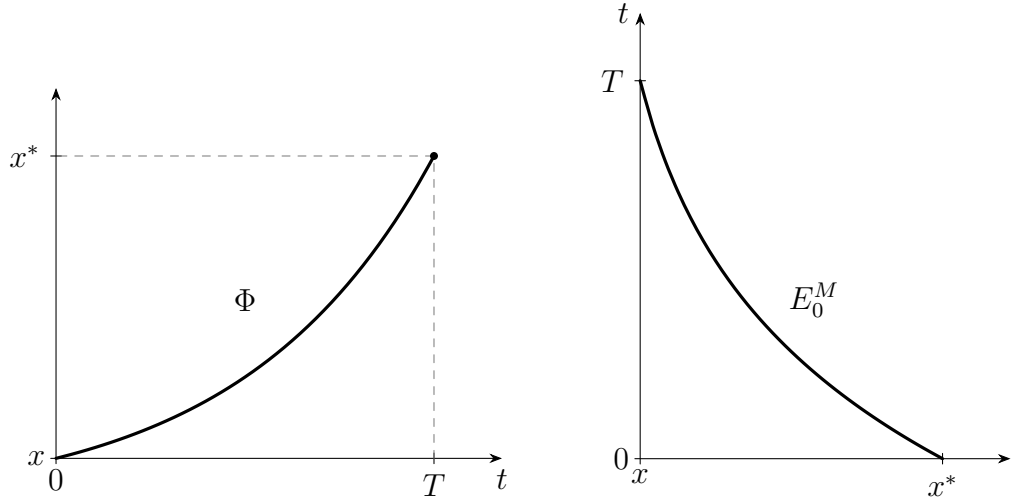
\begin{figure}[hbt]
    \centering
    \begin{tikzpicture}[>=Stealth,line cap=round]
  \def\T{5}      
  \def\X{4}      

  \draw[->] (-2pt,0) -- (\T+0.9,0) node[below] {$t$};
  \draw[->] (0,-2pt) -- (0,\X+0.9) node[left]  {};

  \draw (\T,2pt) -- (\T,-2pt) node[below] {$T$};
  \draw (2pt,\X) -- (-2pt,\X) node[left]  {$x^*$};
  \node[below] at (0,0) {$0$};      
  \node[left]       at (0,0) {$x$};   

  \draw[dashed,gray] (\T,0) -- (\T,\X) -- (0,\X);

  \draw[very thick, domain=0:\T, smooth,variable=\t]
        plot ({\t},{\X*(exp(2*\t/\T)-1)/(exp(2)-1)});
  \node at (2.5,2.1) {$\Phi$};

  \fill (\T,\X) circle (1.5pt);
\end{tikzpicture}
\qquad
\begin{tikzpicture}[>=Stealth,line cap=round]
  \def\T{5}      
  \def\X{4}      

  \draw[->] (-2pt,0) -- (\X+0.9,0) node[below] {};
  \draw[->] (0,-2pt) -- (0,\T+0.9) node[left]  {$t$};

  \draw (\X,2pt) -- (\X,-2pt) node[below] {$x^*$};
  \draw (2pt,\T) -- (-2pt,\T) node[left]  {$T$};
  \node[below] at (0,0) {$x$};      
  \node[left]       at (0,0) {$0$};   

  \draw[very thick,domain=0:\X, smooth, variable=\x]
      plot ({\x},{\T - (\T/2)*ln(1 + \x*(exp(2)-1)/\X)});
  \node at (2.3,2.1) {$E^M_0$};
\end{tikzpicture}
    \caption{Left: The flow of a differential equation with strictly monotonically increasing vector field. Right: Deterministic escape time function $E^M_0$ for the flow on the left.}
    \label{fig:cartoon}
\end{figure}

The reasoning from~\cite{dorrington2023interaction} can be understood for this situation as follows. 
Let us simplify noise to act only at one time instant by offsetting the trajectory by $\pm \delta$ with equal probability. In Figure~\ref{fig:cartoon}, if the trajectory is at the point $(t,y)$ on the graph of $\Phi$ when the perturbation occurs, it can be thought of as sliding the point along the graph by a vertical displacement of~$\pm \delta$. Now, since the graph is convex, a negative vertical displacement implies a larger negative shift horizontally than the positive horizontal shift that a positive vertical displacement would imply. Hence, in expectation, the noise induces an increase in the time to reach~$x^*$.

This latter, intuitive view is no longer applicable in higher dimensions. As we show in the next section, the question of stochastic inertia for higher-dimensional unstable hyperbolic fixed points is not trivial, but it can be analyzed by the formula given in theorems~\ref{thm:stoch_inertia} and~\ref{thm:derivative}.

\subsection{Structure}

The remainder of the work is structured as follows.

In Section~\ref{sec:examples} we show theorems~\ref{thm:stoch_inertia},~\ref{thm:derivative}, and, in particular, formula~\eqref{eq:thm} in action. To this end, in Section~\ref{sec:computing} we derive a computable formula for evaluating~\eqref{eq:thm} whenever $E_0^M$ is not known analytically; only $v$ and~$G$ are, together with a smooth function $N:\R^d\to\R$ such that~$M = \{N<0\}$. Then, in Section~\ref{sec:fixed_point}, we consider the natural case when $M$ is some neighborhood of an unstable fixed point. Section~\ref{ssec:relaxation_oscillator} shows how our results indicate stochastic inertia in relaxation oscillators, as showcased on the van der Pol system.
In Section~\ref{ssec:discontinuity} we demonstrate the role of Assumption~\ref{ass:no_stable}: $\sigma \mapsto E_\sigma^M(x)$ need not be continuous at~$\sigma=0$, let alone differentiable, if the deterministic system fails to exit $M$ from an open set of initial conditions.

Section~\ref{sec:proofs} contains the proofs of theorems~\ref{thm:stoch_inertia} and~\ref{thm:derivative}. The idea of the proof, for simplicity, in the case of~$G = \Id$, is to use Dynkin's formula to obtain
\begin{equation}
    \label{eq:Dynkin_sketch}
    \begin{aligned}
    \E_x\Big[E_0^M(X^{\tau_\sigma^M}_\sigma)\Big] &= E_0^M(x) + \E_x\left[ \int_0^{\tau_\sigma^M} (\sigma \Delta + v\cdot \nabla) E_0^M(X^t_\sigma)\, \dd t \right] \\
    &= E_0^M(x) + \sigma \E_x\left[ \int_0^{\tau_\sigma^M} \Delta E_0^M(X^t_\sigma)\, \dd t \right] - \E_x[\tau_\sigma^M],
    \end{aligned}
\end{equation}
where we used $v \cdot \nabla E_0^M = -1$ to arrive at the second line. Observing that $E_\sigma^M(x) = \E_x\big(\tau_\sigma^M\big)$, rearranging~\eqref{eq:Dynkin_sketch} would yield~\eqref{eq:thm} if the terms involved behaved in the $\sigma\to 0$ limit as naturally expected: $\smash{ \tau_\sigma^M \to E_0(x)}$, $ X_\sigma^t \to \phi^t x$ in an appropriate sense and~$\smash{\sigma^{-1}\E_x\big[E_0^M(X^{\tau_\sigma^M}_\sigma)\big] \to 0}$. A localization of~\eqref{eq:Dynkin_sketch} around the deterministic trajectory, together with quantitative estimates enabled by the Freidlin--Wentzell theory~\cite{freidlin1998random}, leads to the desired conclusions.

Finally, Section~\ref{sec:outlook} discusses possible future directions, in which sense and how the pointwise theory presented here could be lifted to a single quantity related to $M$ itself, describing an instance of stochastic inertia. While Section~\ref{ssec:regime_average} considers general difficulties that arise in this context, Section~\ref{ssec:limit_cycle} focuses specifically on stochastic limit cycles.

\section{Examples of stochastic inertia}
\label{sec:examples}

\subsection{Computing the integral of $\Gamma$ along a trajectory}\label{sec:computing}

The central quantity that appears in the formulas of Theorem \ref{thm:stoch_inertia} and Theorem \ref{thm:derivative} is
\begin{equation}
    \int_0^{E_0^M(x)} \Gamma(\phi^t x)\,\dd t,
\end{equation}
where
\begin{equation}\label{eq:recall_Gamma}
        \Gamma(x)
    = \operatorname{trace} \left( G(x)^\top \left[\operatorname{Hess} E_0^M(x) \right]  G(x)\right).
\end{equation}
In the case of isotropic additive noise, that is $G(x) = \Id$, the term $\Gamma(x)$ is simply $\Delta E_0^M(x)$.
In this section, we present an algorithmic method to compute this integral for a given starting point. The formulas that we arrive at are justified in Appendix \ref{app:derivation}.

We assume that $x\in M$ leaves the domain non-tangentially in a point $\theta \in \partial M$ and that
$\partial M$ is twice continuously differentiable around $\theta$. By Theorem \ref{thm:stoch_inertia}, we know that $E_0^M$ is twice differentiable along the deterministic trajectory. We first compute this trajectory up to the exit point, and then compute the gradient and Hessian of $E_0^M$ backward along it. The integral is computed together with these derivatives.

To compute the Hessian of $E_0^M$ at points along the trajectory, we need a way to characterize the boundary of $M$. Let $N:\R^d \to \R$ be a function with the following properties:
\begin{enumerate}[(i)]
    \item $N$ is negative inside of $M$, positive outside of $M$ and $N=0$ on the boundary $\partial M$.
    \item $N$ is twice continuously differentiable near $\theta$ and $\nabla N(\theta) \neq 0$.
\end{enumerate}
A natural choice for $N$ would be the signed distance to the boundary $\partial M$, where points inside of $M$ get a negative sign. We emphasize that any choice of $N$ with these properties leads to the same value of the integral. To perform the computation, we need a way to evaluate the following functions:
\begin{itemize}
    \item The velocity field $v$, together with its first two derivatives $Dv$ and $D^2v$;
    \item The diffusion matrix $G$;
    \item A boundary-defining function $N$, together with its first two derivatives $\nabla N$ and $\operatorname{Hess} N$.
\end{itemize}
Given a way to evaluate these functions, we proceed as follows:\\

\noindent
\textbf{1. Find the exit point $\theta$.}
Solve the deterministic equation
\begin{equation}
    \partial_t X^t=v(X^t), \qquad X^0=x,
\end{equation}
until the trajectory first reaches $\partial M$. Since this is also the first time that $N(X^t)=0$, the exit time can be computed accurately using e.g.~Newton's method. Record the exit time $T=E_0^M(x)$, the exit point $\theta=X^T$, and a sequence of points $X^t$ for $0\leq t\leq T$. This sequence of points will be used for a quadrature later on, so they should be chosen sufficiently dense. Every point $X^t$ with $0\leq t<T$ has the same exit point $\theta$ and remaining exit time $T-t$.\\

\noindent
\textbf{2. Evaluate the boundary data.}
Near $\theta$, the boundary is described by the twice differentiable function $N$ with $N=0$ on $\partial M$. Set
\begin{equation}
    n=\nabla N(\theta),\qquad b=v(\theta),\qquad c=n^\top b,
\end{equation}
and write
\begin{equation}
    a=Dv(\theta)^\top n,\qquad P=\Id-\frac{bn^\top}{c}.
\end{equation}
The non-tangential exit assumption ensures that $c>0$. The matrix $P$ maps a displacement to the corresponding change in the exit point. In particular, $P$ is the projection onto the tangent plane of $\partial M$ in $\theta$ along the direction $b$.

The gradient and Hessian of $E_0^M$ at the exit point are given by
\begin{align}\label{eq:boundary_derivatives}
    g(T)&=-\frac{n}{c},\\
    H(T)&=-\frac{P^\top \operatorname{Hess} N(\theta) P}{c}
        +\frac{na^\top+an^\top}{c^2}
        -\frac{a^\top b}{c^3}nn^\top.
\end{align}
If the boundary $\partial M$ is flat around $\theta$, then $N$ can be chosen affine near $\theta$. This implies that $\operatorname{Hess} N(\theta)=0$ and the first term in $H(T)$ vanishes.\\

\noindent
\textbf{3. Backtrack the Hessian of $E_0^M$.}
The vector $g(t)$ and the symmetric matrix $H(t)$ describe the gradient and Hessian of $E_0^M$ at $X^t$. Starting from the boundary data in Step~2, solve
\begin{nalign}[eq:derivative_transport]
    \partial_t g(t)&=-Dv(X^t)^\top g(t),\\
    \partial_t H(t)&=-Dv(X^t)^\top H(t)-H(t)Dv(X^t)
        -\sum_{k=1}^d g_k(t)\operatorname{Hess}v_k(X^t),
\end{nalign}
backward from $t=T$ to $t=0$. Here, $\operatorname{Hess}v_k$ is the Hessian of the $k$-th component of the velocity field. The coefficients are evaluated using the trajectory recorded in Step~1.\\

\noindent
\textbf{4. Combine the terms and integrate.}
The integrand at time $t$ is given by
\begin{equation}\label{eq:Gamma_computation}
    \Gamma(X^t)=\operatorname{trace}\left(G(X^t)^\top H(t)G(X^t)\right).
\end{equation}
In the previous steps, we have computed $H(t)$ at a sequence of times $0\leq t\leq T$, which allows us to compute $\Gamma(X^t)$ along these times. Numerical quadrature yields an approximation to the integral
\begin{equation}
    \int_0^{E_0^M(x)}\Gamma(\phi^t x)\,\dd t.
\end{equation}

\subsection{Stochastic inertia around fixed points}\label{sec:fixed_point}

Assume that the deterministic dynamics has a fixed point $x_0\in \R^d$, i.e.~$v(x_0)=0$, and let $M\subset \R^d$ be an open neighbourhood of $x_0$. We want to study the expected escape time $E_\sigma^M(x)$ for $x\in M$ and determine whether there is stochastic inertia. Note that in order for $E_0^M(x)$ to be finite around $x_0$, the fixed point must have at least one unstable direction. To simplify the system, we consider the linearization of $v$ around $x_0$. For a linear system it makes sense to assume w.l.o.g.~that $x_0=0$. Hence, we consider the vector field $v(x)=Ax$, for $A\in \R^{d\times d}$ a matrix, and $M\subset\R^d$ an open neighbourhood of $0$. We present three examples which all support the following heuristic.
\begin{heuristic}
    We expect stochastic inertia around a fixed point if there is a single direction that dominantly leads the escape dynamics from the domain.
\end{heuristic}
Example \ref{ex:ball} shows that if there are multiple unstable dimensions, we can in general not expect stochastic inertia. Example \ref{ex:cube} indicates that it is still possible to find stochastic inertia, even when having multiple unstable dimensions. In this example, the shape of $M$ is chosen such that for each starting point only one of the unstable dimensions is significant for the escape from $M$. Lastly, Example \ref{ex:dominant} shows that for fixed points with a dominant unstable direction, we find stochastic inertia for initial conditions close to the origin, regardless of the shape of $M$.

\begin{example}\textbf{(Escape from a ball with $\mathbf{d}$ unstable dimensions)} \label{ex:ball}

    Let $\lambda >0$ and consider the linear vector field given by the matrix $A=\lambda \Id\in\R^{d\times d}$. The vector field corresponds to uniform expansion at rate $\lambda$ such that the origin is an unstable fixed point with $d$ unstable dimensions. Let $r>0$ and let $M=B_r(0)$ be the open ball around $0$ with radius $r$. Note that $\norm{\phi^t x} = e^{\lambda t} \norm{x}$. For $x\in M$, we compute
    \begin{equation}
        E_0^M(x) = \inf\{t>0\mid \lVert \phi^tx\rVert \geq r \} = \lambda^{-1} \big( \log(r) - \log\lVert x \rVert\big).
    \end{equation}
    Any $x^*\in M \setminus \{0\}$ leaves $M$ non-tangentially in a finite time. Additionally $\partial M$ is smooth everywhere and it is easy to check that the linear vector field $v$ satisfies Assumption \ref{ass:no_stable}. We choose isotropic additive noise, that is $G(x) = \Id$, which naturally satisfies Assumption \ref{ass:uniformly elliptic}. Hence, the derivative $\mathcal{D}(x^*)$ of $E_\sigma^M(x^*)$, as defined in \eqref{eq:derivative_sigma}, can be computed using Theorem \ref{thm:derivative} and determines whether there is stochastic inertia in $x^*$. We compute
    \begin{align}
        E_0^M(x) &= \lambda^{-1} \log(r)- \frac{1}{2} \lambda^{-1} \log\norm{x}^2,\\
        \nabla E_0^M(x) &=- \lambda^{-1} \norm{x}^{-2} x, \\
        \Delta E_0^M (x) &= -\lambda^{-1}\sum_{i=1}^d \big(-2\norm{x}^{-4} x_i^2 + \norm{x}^{-2} \big) \\
        &= \lambda^{-1} (2-d) \norm{x}^{-2}.
    \end{align}
    By Theorem \ref{thm:derivative} the derivative $\mathcal{D}(x^*)$ is given by
    \begin{equation}
        \mathcal{D}(x^*) = \int_0^{E_0^M(x^*)} \Delta E_0^M(\phi^tx^*) \,dt = \lambda^{-1} (2-d) \norm{x^*}^{-2} \int_0^{E_0^M(x^*)}e^{-2\lambda t}  \,dt.
    \end{equation}
    Note that the sign of $\mathcal{D}(x^*)$ depends solely on the dimension $d$. For $d=1$, there is always stochastic inertia, while for $d>2$, there never is stochastic inertia. For $d= 2$, we have $\mathcal{D}(x^*) = 0$ for all $x^*\in M\setminus \{0\}$, and we cannot make an a priori statement about stochastic inertia (further analysis shows that there is no stochastic inertia for this system in two dimensions). Additionally, we observe that for $d=1$, the derivative $\mathcal{D}(x^*)$ gets larger for initial conditions close to $0$. This agrees with the heuristic derivation of \cite{dorrington2023interaction} which states that the effect of stochastic inertia is stronger close to an unstable fixed point. \hfill $\Diamond$
\end{example}

\begin{example}\textbf{(Escape from a cube with $d$ unstable directions)} \label{ex:cube}

    As in the previous example, let $\lambda>0$, define $A=\lambda \Id \in \R^{d\times d}$, and choose isotropic additive noise $G(x) = \Id$. This time we consider escape from the open cube $M=(-r,r)^d$ for $r>0$. Since $\phi^t x = e^{\lambda t} x$, the deterministic escape time is given by
    \begin{equation}
        E_0^M(x) = \lambda^{-1} \big(\log(r) - \log\big(\max_{i=1,\hdots,d} \abs{x_i}\big) \big).
    \end{equation}
    Let $x^* \in M \setminus \{0\}$ for which there is a unique $j\in\{1, \hdots, d\}$ such that $|x^*_j| = \max_{i=1,\hdots,d} \abs{x_i^*}$. Otherwise, the deterministic trajectory of $x^*$ leaves $M$ in a corner, in which $\partial M$ is not differentiable. Hence, $x^*$ leaves $M$ non-tangentially. In a small neighbourhood of $x^*$, the deterministic escape time only depends on the $j$-th coordinate. We conclude
    \begin{align}
        E_0^M(x) &= \lambda^{-1} \big( \log(r) - \log|x_j|\big), \\
        \nabla E_0^M(x) &= - \lambda^{-1} x_j^{-1} e_j, \\
        \Delta E_0^M(x) &= \lambda^{-1} x_j^{-2}.
    \end{align}
    Since $\phi^t$ scales all dimensions uniformly, the $j$-th coordinate is by absolute value the largest coordinate of $\phi^tx^*$ for all $t\geq 0$. We compute $\mathcal{D}(x^*)$ with the help of Theorem \ref{thm:derivative}.
    \begin{equation}
        \mathcal{D}(x^*) = \int_0^{E_0^M(x^*)} \Delta E_0^M(\phi^tx^*) \,dt = \lambda^{-1}  (x_j^*)^{-2} \int_0^{E_0^M(x^*)}e^{-2\lambda t}  \,dt.
    \end{equation}
    This shows that $\mathcal{D} (x^*)>0$, and therefore, there is stochastic inertia in all $x^* \in M \setminus \{0\}$ for which there is a unique largest coordinate. This holds independently of the number of dimensions $d$.   \hfill $\Diamond$
\end{example}

The only difference between Example \ref{ex:ball} and Example \ref{ex:cube} is the shape of the domain $M$. However, in the former example the existence of stochastic inertia depends on the number of dimensions $d$ while in the second example the number of dimensions does not matter. This is because for the escape from a cube, only the value of one coordinate is relevant for the escape time. Hence, in the small noise limit escape from a cube is effectively a one-dimensional problem, and Example \ref{ex:ball} states that in one dimension there always is stochastic inertia. This supports the heuristic that we expect stochastic inertia if there is a single unstable dimension which is relevant for the escape time. We clarify this heuristic in another example. The essence of the following example is that if there is a dominant unstable dimension, then there is stochastic inertia for most initial conditions close to the origin.

\begin{example}\textbf{(Escape with a dominant unstable direction)}\label{ex:dominant}
    Let $A\in \R^{d\times d}$ be a matrix with a simple dominant eigenvalue, meaning that there is a real simple eigenvalue $\lambda_1>0$ which is larger than the real part of any other eigenvalue. 
    For simplicity,  we choose isotropic additive noise, that is $G(x) = \Id$. 
    Let $M\subset \R^d$ be an open, bounded neighborhood of the origin. We argue that there is stochastic inertia in almost all initial conditions sufficiently close to the origin. To keep the argument short, we deliberately give only a heuristic argument rather than a rigorous proof.
     
    Let $r_1$ and $\ell_1$ be the right and left eigenvectors of the matrix $A$ corresponding to the dominant eigenvalue $\lambda_1$. We scale $\ell_1$ and $r_1$ such that $\abs{r_1}=1$ and $\ell_1^\top r_1 = 1$. The matrix $\Pi_1 \coloneqq r_1 \ell_1^\top$ is the linear projection onto $r_1$ along the subspace of vectors that grow slower than $e^{t\lambda_1}$ under $\exp(tA)$. Let $x\in \R^d$ with $\Pi_1 x \neq 0$. For $\epsilon>0$ sufficiently small, the initial condition $\epsilon x $ lies inside of $M$. We compute the derivative $\mathcal{D}(\epsilon x)$ using the scheme derived in Section \ref{sec:computing}.

    The trajectory of the initial condition $\epsilon x$ is given by $\phi^t (\epsilon x) = \epsilon \exp(tA) x$. In particular, there is the splitting
    \begin{equation}\label{eq:splitting}
        \phi^t(\epsilon x) = \epsilon e^{\lambda_1 t} \Pi_1 x + \epsilon \exp(tA) (\Id-\Pi_1) x.
    \end{equation}
    As $\epsilon$ goes to 0, the deterministic escape time $T_\epsilon\coloneqq E_0^M(\epsilon x)$ grows to $\infty$. For large $t$, the ratio of the dominant eigenvalue of $\exp(tA)$ to the rest of the spectrum grows exponentially. Hence, the vector $\phi^t (\epsilon x)$ quickly aligns with the direction of~$r_1$. In particular, the exit point $\theta_\epsilon$ is exponentially close to a multiple of $r_1$ and the exit time $T_\epsilon$ is of order $-\lambda_1^{-1}\log(\epsilon \abs{\Pi_1 x})$.
    
    Let $\epsilon$ be sufficiently small such that $\theta_\epsilon$ is well approximated by a multiple of $r_1$. In the following, we omit the dependence on $\epsilon$. Recall the notation from Step 2 of Section \ref{sec:computing},
    \begin{equation}
        n=\nabla N(\theta),\qquad b=v(\theta),\qquad c=n^\top b,\qquad
        a=Dv(\theta)^\top n,\qquad P=\Id-\frac{bn^\top}{c}.
    \end{equation}
    Using this notation, the Hessian of $E_0^M$ at $\theta$ is given by \eqref{eq:boundary_derivatives},
    \begin{align}\label{eq:recall_H_T}
        H(T) =-\frac{P^\top \operatorname{Hess} N(\theta) P}{c}
        +\frac{na^\top+an^\top}{c^2}
        -\frac{a^\top b}{c^3}nn^\top.
    \end{align}
    Under the assumption that $\theta$ is approximately a multiple of $r_1$, one can verify that
    \begin{equation}\label{eq:approximation_HT}
        r_1^\top H(T)\, r_1 \approx \frac{1}{\lambda_1 \abs{\theta}^2}>0.
    \end{equation}
    Indeed, the first term of \eqref{eq:recall_H_T} vanishes since $P\theta \approx 0$. The second term is approximately $2\lambda_1^{-1}\abs{\theta}^{-2}$ while the last term is approximately $-\lambda_1^{-1}\abs{\theta}^{-2}$. We conclude that for $\epsilon$ sufficiently small, the term $r_1^\top H(T) r_1 >0$.
    
    Since the vector field $v(x) = Ax$ is linear, the evolution of the Hessian $H(t)$ in \eqref{eq:derivative_transport} is
    \begin{align}
        \partial_t H(t) = -A^\top H(t) - H(t) A.
    \end{align}
    Given the terminal matrix $H(T)$, the solution for $0\leq t \leq T$ is computed by 
    \begin{equation}
        H(t) = \exp((T-t) A)^\top H(T) \exp((T-t) A).
    \end{equation}
    Since the diffusion matrix is chosen to be the identity, $G = \Id$, we have $\Gamma(\phi^t(\epsilon x)) = \operatorname{trace}(H(t))$. For $(T-t)$ large, the matrix $\exp((T-t) A)$ has an exponentially dominant eigenvalue $e^{\lambda_1(T-t)}$. Thus, the dominant term in the trace of $H(t)$ is
    \begin{equation}
        \sum_{k=1}^d e_k^\top \Pi_1^\top \exp((T-t)A)^\top H(T) \exp((T-t)A) \Pi_1 e_k.
    \end{equation}
    Note that $\Pi_1 e_k = r_1 \ell_1^\top e_k$. Using the approximation \eqref{eq:approximation_HT}, the dominant term in the trace is
    \begin{equation}
        \sum_{k=1}^d (\ell_1^\top e_k)^2 e^{2\lambda_1 (T-t)} r_1^\top H(T)\, r_1 \approx e^{2\lambda_1 (T-t)} \frac{\abs{\ell_1}^2}{\lambda_1 \abs{\theta}^2} >0.
    \end{equation}
    This term is positive and grows exponentially in $(T-t)$. As $\epsilon$ goes to $0$, the integral of $\Gamma(\phi^t(\epsilon x))$ predominantly integrates over times $t$ for which $(T-t)$ is large and the integral becomes increasingly positive, indicating stochastic inertia.    
    \hfill $\Diamond$
\end{example}

\begin{remark}\label{rmk:finite_observability}
    This section presented examples of linear systems for which the derivative $\mathcal{D}(x)$ is positive, indicating stochastic inertia. This, however, makes no statement about how much the average escape time increases under the addition of noise. It might be the case that $\mathcal{D}(x)$, the first $\sigma$-derivative of $E_\sigma^M(x)$, is large, but the second derivative is extremely negative. Then, the total increase in the average escape time due to the addition of noise is small, and stochastic inertia is only observable for very small values of $\sigma$. Exactly this phenomenon occurs close to the origin (the deterministic fixed point) of several examples above. In these examples, the derivative $\mathcal{D}(x)$ grows as $x$ approaches the origin, indicating strong stochastic inertia. On the other hand, the range of $\sigma$ for which $E_\sigma^M(x)>E_0^M(x)$ shrinks as $x\to 0$. This is because even small noise accelerates the escape from a fixed point (which has an infinite deterministic escape time). A similar phenomenon occurs when we perturb the boundary of $M$. Comparing Example \ref{ex:ball} and Example \ref{ex:cube} shows that the curvature of the boundary can influence the sign of $\mathcal{D}(x)$. Thus, even small perturbations of $\partial M$ can result in large changes of $\mathcal{D}(x)$. On a macroscopic level, i.e., for $\sigma$ not infinitesimally small, such boundary perturbations should have little influence on the average escape. We conclude that while $\mathcal{D}(x)>0$ proves the existence of a range $(0,\sigma^*)$ for which $E_\sigma^M(x) >E_0^M(x)$, additional information is needed to determine how large $\sigma^*$ is. 
\end{remark}

\subsection{Relaxation oscillators}
\label{ssec:relaxation_oscillator}

Relaxation oscillations are widely studied phenomena in multiple time scale dynamical systems~\cite[Section~5.2]{KuehnBook}. They describe periodic behavior where slow relaxation at quasi-equilibria alternates with fast excursions. These fast excursions typically stem from an instability arising from the passage around a non-hyperbolic fold point (in contrast to the hyperbolic fixed points from the previous section), or, put differently, the slow dynamics follow a normally hyperbolic regime that is lost near a singularity that produces fast outbursts.

A prototypical relaxation oscillator is given by the van der Pol equations~\eqref{eq:van_der_pol}, introduced in~\cite{van1920theory} in the context of electrical circuits. These contain a parameter $\varepsilon>0$ separating time scales and giving rise to the ``relaxation'' and ``impulsive'' sections of the limit cycle. 
The interaction between the parameter $\varepsilon \ll 1$ and the noise strength $\sigma$ can induce various kinds of behaviors; see~\cite[Section~6.1.2]{BerglundGentz2006} for an analysis near the bifurcation points, and, in particular, \cite{BerglundGentzKuehn} for noise-induced earlier escape around a fold singularity, as present in relaxation oscillators, in a three-dimensional fast-slow system with noise rescaled also in terms of $\varepsilon$.

We focus on the basic dynamics of the
planar van der Pol oscillator, given by the following ODE on $\R^2$
\begin{nalign}[eq:van_der_pol]
    \dot x &= y, \\
    \dot y &= \tfrac{1}{\varepsilon}(1-x^2)y - x.
\end{nalign}
Writing \eqref{eq:van_der_pol} as $\ddot x - \tfrac{1}{\varepsilon}(1-x^2)\dot x + x = 0$
and introducing the Li\'enard variable $w = \varepsilon y - x + \tfrac{x^3}{3}$ gives the
slow--fast system
\begin{nalign}[eq:van_der_pol_lienard]
    \varepsilon \dot x &= w + x - \tfrac{x^3}{3}, \\
    \dot w &= -\varepsilon x,
\end{nalign}
with critical manifold
\[
    S_0 = \left\{ (x,w) : w = \tfrac{x^3}{3} - x \right\}.
\]
Normal hyperbolicity is lost where $\partial_x\!\left(w + x - \tfrac{x^3}{3}\right) = 1 - x^2 = 0$,
that is, at $x = \pm 1$. The two fold points are therefore
\[
    (x,w) = \left(1, -\tfrac{2}{3}\right) \quad \text{and} \quad (x,w) = \left(-1, \tfrac{2}{3}\right),
\]
which separate the attracting branches ($|x| > 1$) from the repelling branch ($|x| < 1$).
In the original coordinates, $\varepsilon y = w + x - \tfrac{x^3}{3}$ vanishes at both points,
so the fold points are $(x,y) = (\pm 1, 0)$.

We set $\varepsilon=0.1$, and consider two domains $M_\textnormal{s}$ and $M_\textnormal{f}$ around one of the slow segments and around one of the fast segments, respectively; see Figure~\ref{fig:phase_portrait}. Due to the shape of the respective segments, we choose $M_\textnormal{s}$ to be a rectangle in original coordinates, while we choose $M_\textnormal{f}$ to be a rectangle in $(x,w)$-coordinates. Let $\Phi_\varepsilon:\R^2 \to \R^2$ be the coordinate transform from $(x,y)$ to $(x,w)$, that is
\begin{equation}
    \Phi_\varepsilon(x,y) = \big(x,\varepsilon y - x + \tfrac{x^3}{3}\big).
\end{equation}
This map is a bijection. We define
\begin{align}
    M_\textnormal{s}&=(-2.1,-0.8)\times(-2,2), \\ M_\textnormal{f}&=\Phi_\varepsilon^{-1}\left((-0.5,1.1)\times(-0.95,-0.5)\right).
\end{align}
On $M_\textnormal{s}$, the dynamics are attracted to the slow manifold which then governs the dynamics before the fold point is approached. The regime $M_\textnormal{f}$ includes the fold point $(1, 0)$ and parts of the fast segment after the fold point.
\begin{figure}%
    \centering
    \includegraphics[width=0.45\linewidth]{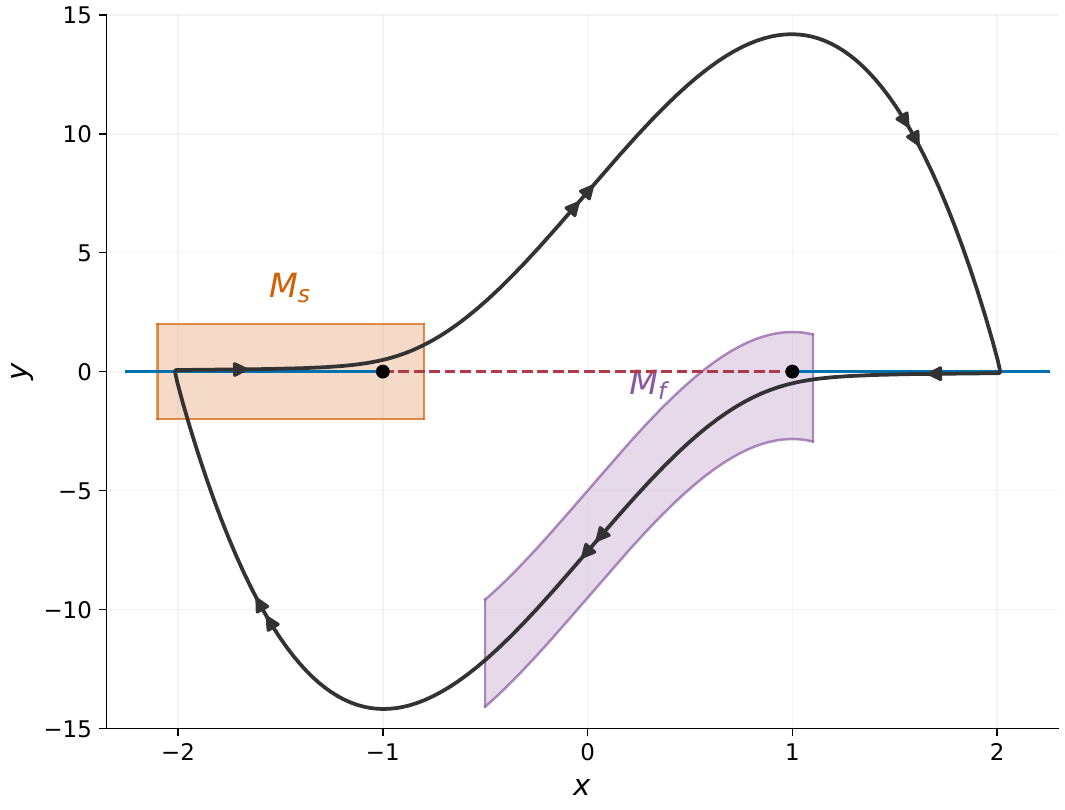}
    \hspace{0.05\linewidth} 
    \includegraphics[width=0.45\linewidth]{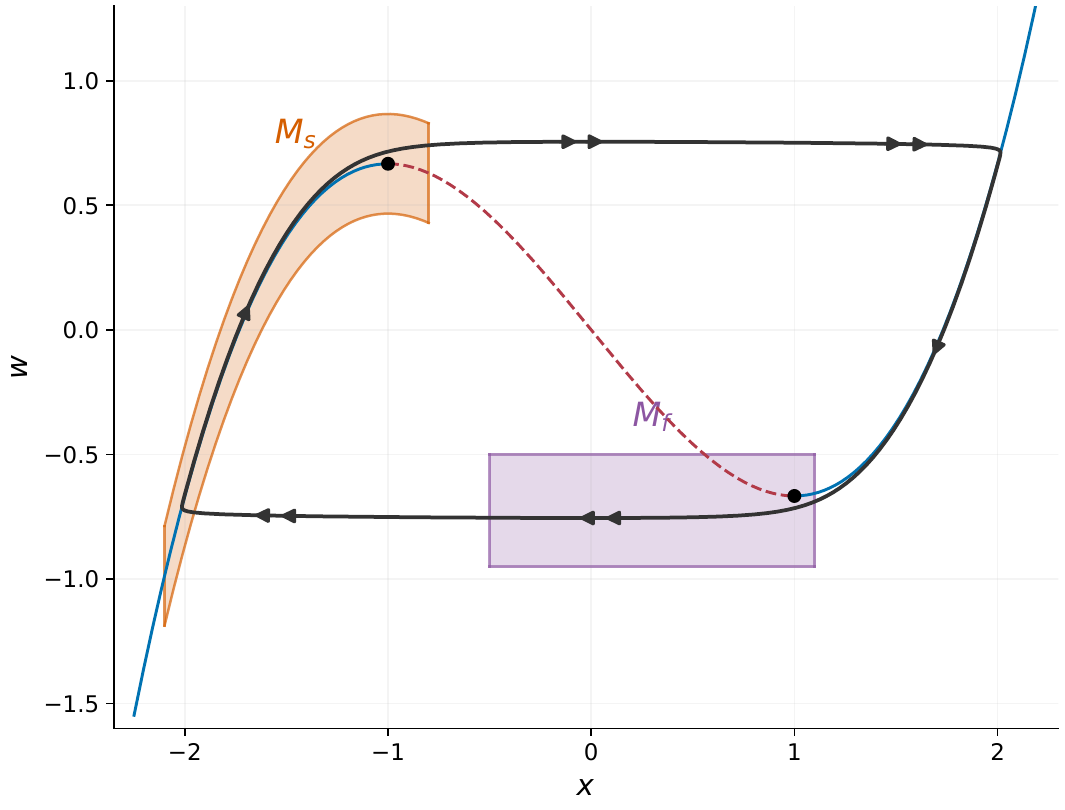}
    \caption{Limit cycle of the van der Pol oscillator. The left image is in original coordinates and corresponds to \eqref{eq:van_der_pol}. The right image is in Li\'enard-coordinates and corresponds to \eqref{eq:van_der_pol_lienard}. The stable branch of the critical manifold is drawn as a solid blue line and the unstable branch of the critical manifold is drawn as a dashed red line. A single arrow indicates a slow region and a double arrow indicates a fast region. The regimes $M_\textnormal{s}$ and $M_\textnormal{f}$ are chosen to be rectangles in their respective coordinates around one of the slow sections and around one of the fast sections of the dynamics. The fold points are marked as black dots.}
    \label{fig:phase_portrait}
\end{figure}
\begin{figure}
    \centering
    \includegraphics[width=0.95\linewidth]{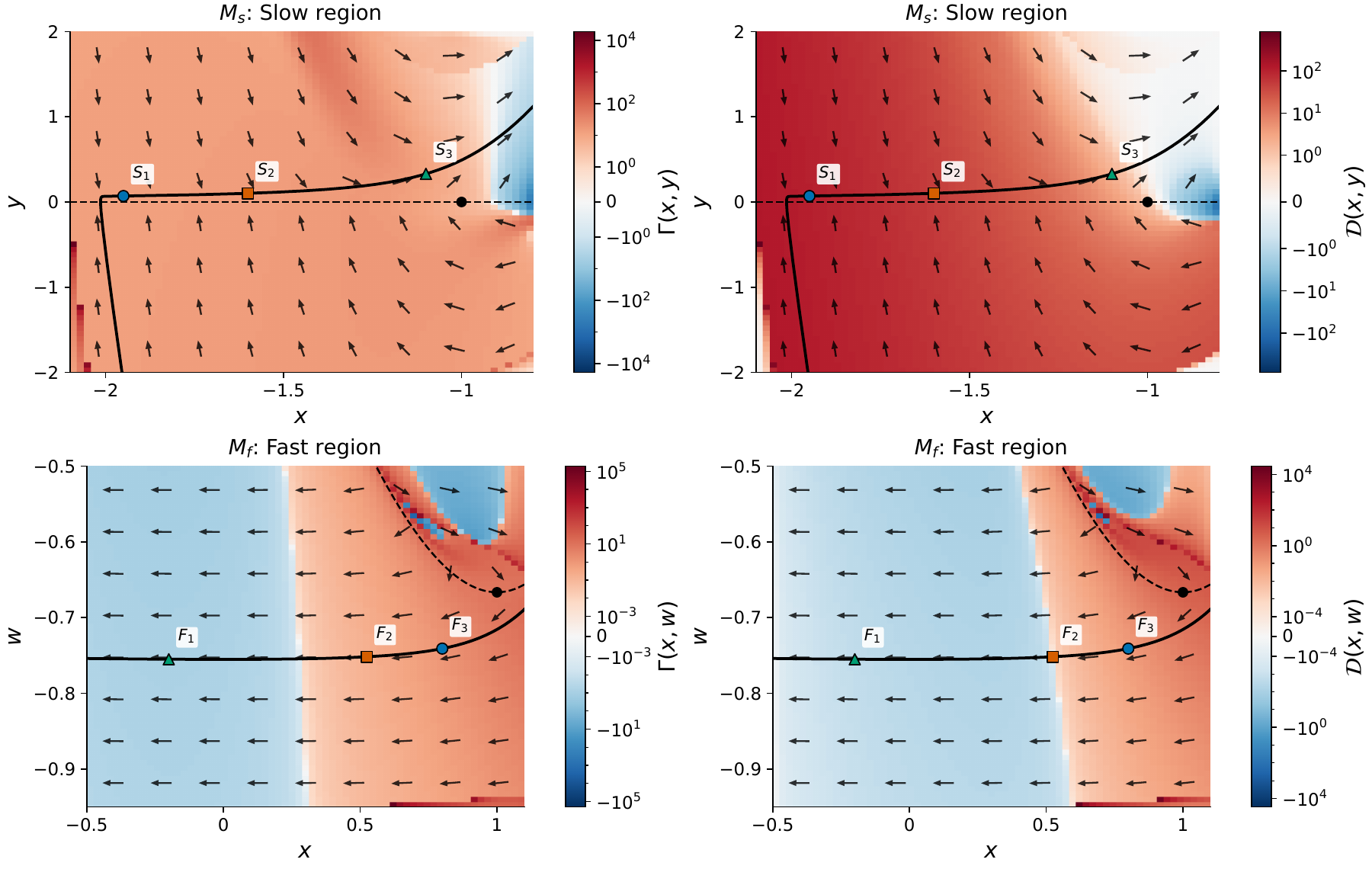}
    \caption{Heatmap representations of the functions $\Gamma$ and $\mathcal{D}$ over the domains~$M_\textnormal{s}$ (top) and $M_\textnormal{f}$ (bottom). Note the different signed logarithmic scalings in the colorbars between the top and the bottom row. Also, note the changed aspect ratio compared to Figure~\ref{fig:phase_portrait}. The small arrows indicate the direction of the velocity field. The limit cycle is shown as a solid black line. The critical manifold $S_0$ is drawn as a dashed line. The points $S_i$ and $F_i$ are used as initial points for a Monte Carlo approximation of the average escape times in Figure~\ref{fig:escape_times}.}
    \label{fig:heatmaps}
\end{figure}

\begin{figure}
    \centering
    \includegraphics[width=0.45\linewidth]{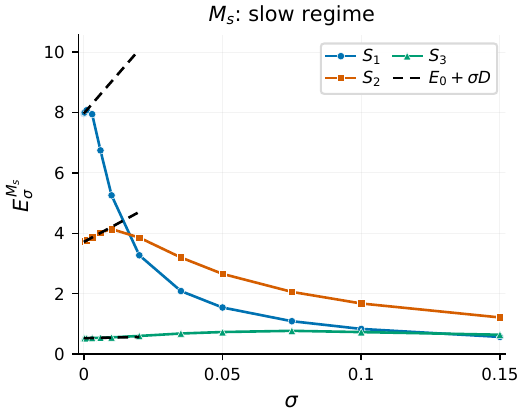} \hspace{0.05\linewidth}
    \includegraphics[width=0.45\linewidth]{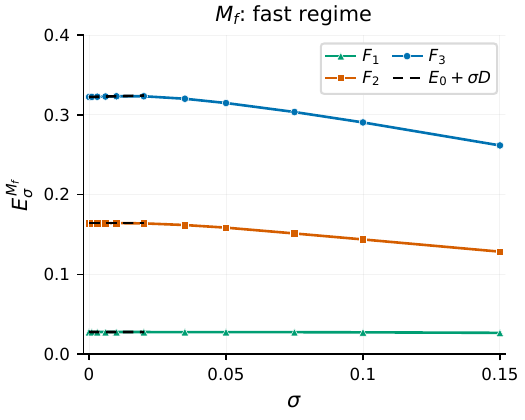}
    \caption{Approximation of the average exit time $E_\sigma^M$ using a Monte Carlo simulation of $5\times10^4$ realisations, for the slow regime $M_\textnormal{s}$ (left) and the fast regime $M_\textnormal{f}$ (right). The initial conditions $S_i$ and $F_i$ are the points shown in Figure~\ref{fig:heatmaps}. For each point, a dashed tangent line starting at $E_0^M$ with slope $\mathcal{D}$ is shown.}
    \label{fig:escape_times}
\end{figure}

We add isotropic additive noise to \eqref{eq:van_der_pol}, turning it into an SDE
\begin{equation}\label{eq:vdp_sde}
    \dd \begin{pmatrix}
        x \\ y
    \end{pmatrix} = \begin{pmatrix}
        y \\ \varepsilon^{-1}(1-x^2)y - x
    \end{pmatrix}\, \dd t + \sqrt{2 \sigma} \, \dd W^t.
\end{equation}
Note that for $M_\textnormal{f}$ we cannot simply work in $(x,w)$-coordinates, since the It\^o-correction term makes the velocity field $\sigma$-dependent. For this reason, we restrict our computations to $(x,y)$-coordinates but display the results for $M_\textnormal{f}$ in $(x,w)$-coordinates.

The velocity field and the diffusion matrix of the SDE \eqref{eq:vdp_sde} satisfy Assumptions~\ref{ass:no_stable} and~\ref{ass:uniformly elliptic}. Hence, Theorem \ref{thm:derivative} provides a formula that determines whether there is stochastic inertia in a given initial condition $X^0$ in $M_\textnormal{s}$ or $M_\textnormal{f}$. For both regimes, we compute the function $\Gamma$, as well as the derivative $\mathcal{D}$, using the method described in Section~\ref{sec:computing}. The results are depicted in Figure~\ref{fig:heatmaps}. 

For the slow regime the pictures show that the derivative $\mathcal{D}$ is positive for most points. This is in accordance with the heuristic of Section \ref{sec:fixed_point}: there is a single direction that dominantly leads the escape dynamics from the domain. After the fold point, $\Gamma$ becomes negative, and there is no longer stochastic inertia. For the fast regime, there is effectively no stochastic inertia. Most points along the limit cycle have values of $\Gamma$ that are close to~$0$. Additionally, since the dynamics is faster, the integrated quantity $\mathcal{D}$ is generally smaller than in the slow regime. In comparison to the slow regime, many points in $M_\textnormal{f}$ take negative values of $\mathcal{D}$, indicating the absence of stochastic inertia. 

To verify our results, we approximate $E_\sigma^M$ for three initial conditions in $M_\textnormal{s}$ and $M_\textnormal{f}$, respectively, with a Monte Carlo simulation. For 12 values of $\sigma$ between 0 and 0.15, we average over $5\times10^4$ realizations of the random exit time $\tau_\sigma^M$. Each realization is computed using the stochastic Heun integration method with a step size of at most~$5\times 10^{-4}$. The resulting graphs of $\sigma \mapsto E_\sigma^M$ are shown in Figure~\ref{fig:escape_times}. The points $S_2$ and $S_3$ clearly show stochastic inertia. In the point $S_1$, there is an initial increase in the average escape time with a steep slope, but the effect of stochastic inertia is only observable for small values of $\sigma$, see Remark~\ref{rmk:finite_observability} for possible explanations. In the fast regime $M_\textnormal{f}$, only $F_3$ shows signs of stochastic inertia, but the effect is negligible. The other two points $F_1$ and $F_2$ show no stochastic inertia. For all six points, the computed derivatives $\mathcal{D}$ align almost perfectly with the slope of $E_\sigma^M$ in $\sigma = 0$. Thus, the formula for $\mathcal{D}$ is a reliable way to predict stochastic inertia in this example.

The code used to create the figures of this section can be found at
\begin{center}
    \url{https://github.com/RobinChemnitz/Stochastic-Inertia}.
\end{center}

\begin{remark}[Inverse stochastic resonance]
    Relaxation oscillators also often appear in neuroscience to describe the spiking behavior of neurons. It is worth pointing out that in this context, a phenomenon similar to stochastic inertia was observed: For moderate but non-zero levels of noise, the spiking rate of the neurons is decreased. This led to the name \emph{inverse stochastic resonance}~\cite{gutkin2009inhibition,tuckwell2012analysis,yamakou2017simple,kuehn2017quenched,lu2020inverse}.
    It is beyond our scope to analyze precise connections between stochastic inertia and inverse stochastic resonance; however, we outline a possible direction in Section~\ref{ssec:limit_cycle}.
\end{remark}

\subsection{Discontinuity at $\sigma=0$}
\label{ssec:discontinuity}

We demonstrate by virtue of an example that without Assumption \ref{ass:no_stable}, the map $\sigma \mapsto E_\sigma^M(x)$ may have a discontinuity at $\sigma=0$.
\begin{example}\textbf{(Discontinuity of $E_\sigma^M(x)$ at $\sigma=0$)}\label{ex:discontinuity}
    Consider the SDE in a potential landscape on $\R$
    \begin{equation}
        \dd X_\sigma^t = -\nabla U(X_\sigma^t) \, \dd t + \sqrt{2\sigma} \, \dd W^t.
    \end{equation}
    Specifically, we choose the potential
    \begin{equation}\label{eq:potential}
        U(x) := -2x^4 + x^2 + 1.
    \end{equation}
    See Figure \ref{fig:potential} for a graph of $U$. We consider escape from the interval $M:= (-1, 1)$. The potential $U$ has a valley in the interval $[-0.5, 0.5]$, and the vector field $v$ has one stable fixed point in $0$ and two unstable fixed points in $\pm 0.5$. The key values of $U$ are given by
    \begin{equation}
        U(0) = 1, \qquad U(\pm 0.5) = 1.125, \qquad U(\pm 1) = 0.
    \end{equation}
    The deterministic escape time $E_0^M(x)$ is infinite for $x\in [-0.5, 0.5]$ and else finite. 

    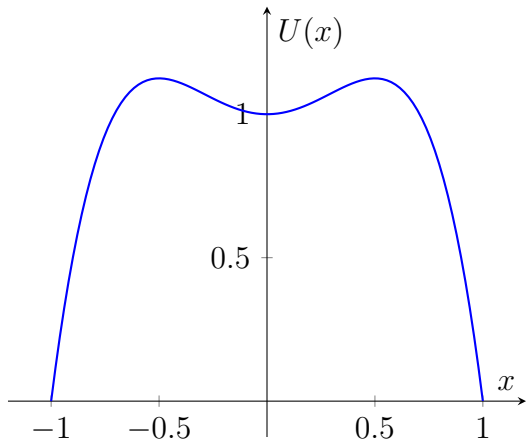
\begin{figure}
        \centering
        \begin{tikzpicture}
        \begin{axis}[
            axis lines=middle,
            xlabel={$x$},
            ylabel={$U(x)$},
            domain=-1:1,
            xmin=-1, xmax=1,
            ymin=0, ymax=1.25,
            samples=100,
            major grid style={line width=.2pt,draw=gray!50},
            enlargelimits=true
        ]
        \addplot[blue, thick] { -2*x^4 + x^2 + 1 };
        \end{axis}
        \end{tikzpicture}
        \caption{The potential $U(x) = -2x^4 + x^2 + 1$.}
        \label{fig:potential}
    \end{figure}

    We investigate the continuity of the map $\sigma \mapsto E_\sigma^M(x)$ in $\sigma = 0$. For $x\in [-0.5, 0.5]$ it is easy to show that $\lim_{\sigma \to 0} E_\sigma^M(x) = \infty$ such that the map is continuous with respect to the compact topology on $\R \cup\{\infty\}$. For $x\in M$ outside of the valley, we use the theory of Freidlin and Wentzell, in particular \cite[Chapter 6, Theorem 5.3]{freidlin1998random} to compute the asymptotic growth of $E_\sigma^M(x)$ as $\sigma \to 0$. Specifically, the theorem provides a formula for
    \begin{equation}
        \lim_{\sigma \to 0} \sigma \log E_\sigma^M(x).
    \end{equation}
    We state the result from the book and provide some intuition for the formula, but refer to \cite[Chapter 6]{freidlin1998random} for more detail.
    
    The key quantity in the formula is $V_M(x,A)$, where $x\in M$ and $A\subset \overline{M}$. In the case of a 1-dimensional potential field, this value describes how much potential needs to be overcome to move from $x$ to the set $A$, while staying inside of $M$, that is
    \begin{equation}
        V_M(x,A) = \inf_{\psi:x\to A} \sup_{t\in [0,1]} U(\psi(t)) - U(x),
    \end{equation}
    where the infimum is taken over all continuous paths $\psi:[0,1]\to \overline{M}$ with $\psi(0)=x$ and $\psi(1) \in A$. In our setting, the formula of \cite[Chapter 6, Theorem 5.3]{freidlin1998random} reads
    \begin{equation}\label{eq:growth_rate_E}
        \lim_{\sigma \to 0} \sigma \log E_\sigma^M(x) = V_M(0, \partial M) - \min\{ V_M(x,0), V_M(0, \partial M)\}.
    \end{equation}

    Note that the right-hand side is always non-negative. This formula has a heuristic interpretation. Let $B:= (-0.5, 0.5)$ be the valley of the potential. The quantity $ V_M(0, \partial M)$ is the exponential rate in $\sigma^{-1}$ at which the expected escape time from the valley grows for $X_\sigma^0 = 0$, i.e.
    \begin{equation}
         V_M(0, \partial M) = \lim_{\sigma \to 0} \sigma \log \E_0[\tau_\sigma^B].
    \end{equation}
    On the other hand, let $V_M(x,0)$ be the rate at which the probability of entering the valley decays for $X_\sigma^0 = x$, i.e.
    \begin{equation}
        V_M(x,0) = -\lim_{\sigma \to 0} \sigma \log \P_x\big(\tau_\sigma^{M\setminus B} < \tau_\sigma^M\big).
    \end{equation}

    The expected escape time from $x\in M$ is at least the probability of entering the valley times the expected escape time from the valley, i.e.~we have \begin{equation}
        \lim_{\sigma \to 0} \sigma \log E_\sigma^M(x) \geq V_M(0, \partial M) - V_M(x,0).
    \end{equation} As long as $V_M(0, \partial M) -  V_M(x,0)$ is non-negative, the lower bound is sharp and \eqref{eq:growth_rate_E} gives the precise exponential growth rate of $E_\sigma^M(x)$.

    For the potential $U$ given by \eqref{eq:potential}, we compute
    \begin{align}
        V_M(0, \partial M)  &= U(\pm0.5) - U(0) = \tfrac{1}{8}, \\
        V_M(x,0) &= \begin{cases}
            0, \quad & |x| < 0.5, \\
            U(\pm0.5) - U(x), \quad &|x| \geq 0.5.
        \end{cases}
    \end{align}
    This lets us compute the asymptotic growth rate of $E_\sigma^M(x)$ with the help of \eqref{eq:growth_rate_E}.
    \begin{equation}
        \lim_{\sigma \to 0} \sigma \log E_\sigma^M(x) = 
        \begin{cases}
        \tfrac{1}{8}, \quad & |x| < 0.5, \\
        U(x) - 1, \quad & |x| \in [0.5, \tfrac{1}{\sqrt{2}}], \\
        0, \quad &|x| > \frac{1}{\sqrt 2}.
        \end{cases} 
    \end{equation}
    This shows that for $x \in M$ with $|x| \in (0.5, \frac{1}{\sqrt 2})$, the deterministic escape time $E_0^M(x)$ is finite, but $E_\sigma^M(x)$ grows exponentially as $\sigma \to 0$. Hence, $\sigma \mapsto E_\sigma^M(x)$ is not continuous in $\sigma=0$ for such $x\in M$. It is possible to show that for $|x|> \frac{1}{\sqrt 2}$ the map $\sigma \mapsto E_\sigma^M(x)$ is continuous. However, this is not the focus of this example. \hfill $\Diamond$
\end{example}

\section{Proofs}\label{sec:proofs}
This section is concerned with the proof of Theorem \ref{thm:stoch_inertia} and Theorem \ref{thm:derivative}. Before we begin with the main part of the proof, we introduce some auxiliary lemmas. Lemma \ref{lem:exp_close_prob} and Lemma \ref{lem:escape_time_convergence} hold without additional assumptions, while Lemma \ref{lem:uniform_in_x} is the key implication of Assumption \ref{ass:no_stable} and Assumption \ref{ass:uniformly elliptic}. All of the results of this section only consider bounded domains of $\R^d$. Therefore, we assume w.l.o.g.~that the velocity field $v$ is globally $C^2$ and the diffusion matrix $G$ is globally Lipschitz.

\subsection{Some helpful lemmas from Freidlin--Wentzell theory}

The first lemma states that as $\sigma \to 0$ the stochastic trajectory converges in probability to the deterministic trajectory uniformly on finite time intervals, and the convergence occurs at an exponential rate which does not depend on the initial condition. The structure of the proof closely follows that of \cite[Chapter 3.1]{freidlin1998random}.

\begin{lemma}\label{lem:exp_close_prob}
   Let $T>0$ and $\epsilon>0$. There are $\gamma>0$ and $C>0$, such that for all $X_\sigma^0 = x\in M$ and $\sigma > 0$, we find 
\begin{equation}\label{eq:P_exp_convergence}
        \P_x\left(\sup_{0\leq t \leq T} |X_\sigma^t - \phi^t x| \geq \epsilon\right) \leq Ce^{-\gamma\sigma^{-1} }.
    \end{equation}
\end{lemma}
\begin{proof}
    By definition of $X_\sigma^t$ and $\phi^tx$, there holds
    \begin{align*}
        X_\sigma^t &= x + \int_0^t v(X_\sigma^s) \; \dd s + \sqrt{2\sigma}\,\int_0^t G(X_\sigma^s) \; \dd W^s,\\
        \phi^tx &= x + \int_0^t v(\phi^s x) \, \dd s.
    \end{align*}
    To simplify notation, we introduce the continuous $\R^d$-valued martingale
    \begin{equation}
        Z^t \coloneqq \int_0^t G(X_\sigma^s) \; \dd W^s.
    \end{equation}
    The difference of the equations of $X_\sigma^t$ and $\phi^t x$ is given by
    \begin{equation}\label{eq:integral_diff}
        X_\sigma^t - \phi^tx = \int_0^t \left[ v(X_\sigma^s) - v(\phi^s x)\right] \, \dd s + \sqrt{2\sigma}\,Z^t.
    \end{equation} 
    The velocity field $v$ is Lipschitz continuous with a constant $L>0$. We conclude
    \begin{equation}\label{eq:inhom_Gronwall}
        |X_\sigma^t - \phi^t x| \leq L \int_0^t  |X_\sigma^s - \phi^s x| \, \dd s +  \sqrt{2\sigma} \left(\sup_{0\leq s \leq T} |Z^s | \right).
    \end{equation}
    A Grönwall-type inequality (cf.~\cite[Chapter 2, Lemma 1.1]{freidlin1998random}) implies that
    \begin{equation}
    \label{eq:BMbound}
        \sup_{0\leq t\leq T} |X_\sigma^t - \phi^t x| \leq \sqrt{2\sigma} \left(\sup_{0\leq t \leq T} |Z^t| \right) e^{L T}.
    \end{equation}
    This inequality provides an upper bound for the probability of $X_\sigma^t$ and $\phi^t x$ differing by at least $\epsilon>0$. 
    \begin{nalign}[eq:lem1_bound]
        \P_x\left( \sup_{0\leq t\leq T} |X_\sigma^t - \phi^t x| \geq \epsilon\right) &\leq \P_x\left(\sup_{0\leq t \leq T} |Z^t| \geq \frac{\epsilon}{\sqrt{2}} e^{-LT}\sigma^{-\frac{1}{2}} \right).
    \end{nalign}
    If the norm of $Z^t$ exceeds $\frac{\epsilon}{\sqrt{2}} e^{-LT}\sigma^{-\frac{1}{2}}$ at some point in time, one of the components of $Z^t$ must exceed $\frac{\epsilon}{d\sqrt{2}} e^{-LT}\sigma^{-\frac{1}{2}}$ at some point in time. The components of $Z^t$ are given by
    \begin{equation}
        Z_i^t = \int_0^t G(X_\sigma^s)_i \, \dd W^s,
    \end{equation}
   where $1\leq i \leq d$ and $G(X_\sigma^s)_i\in \R^d$ is the $i$-th row of the matrix $G(X_\sigma^s)$. The quadratic variation of $Z_i^t$ is given by
   \begin{equation}
       \langle Z_i\rangle_t = \int_0^t \lvert G(X_\sigma^s)_i\rvert^2 \, \dd s.
   \end{equation}
   The matrix $G$ is uniformly bounded and there is a constant $K>0$ such that $\langle Z_i\rangle_t \leq Kt$.
   The exponential martingale inequality, also called Bernstein's inequality (cf.~\cite[Chapter IV]{revuz2013continuous}, in particular Exercise (3.16) therein) states that%
    \begin{equation}
        \P_x\left(\sup_{0\leq t\leq T} |Z_i^t| \geq C \right) \leq 2e^{-\frac{C^2}{2KT}}.
    \end{equation}
    Hence, the probability that the component $Z_i$ exceeds $\frac{\epsilon}{d\sqrt{2}} e^{-LT}\sigma^{-\frac{1}{2}}$ in absolute value at some point in time is upper bounded as
    \begin{align}\label{eq:bound_Z_i}
        \P_x\left( \sup_{0\leq t\leq T} |Z_i^t| \geq \frac{\epsilon}{d\sqrt{2} } e^{-LT}\sigma^{-\frac{1}{2}} \right) \leq 2 e^{-\frac{\epsilon^2}{4KTd^2} e^{-2LT}\sigma^{-1}}.
    \end{align}
    We define the exponent
    \begin{equation}
        \gamma \coloneqq \frac{\epsilon^2}{4KTd^2} e^{-2LT}.
    \end{equation}
    Note that the bound \eqref{eq:bound_Z_i} holds for any $1\leq i \leq d$. By the union bound, the probability of $|Z^t|$ exceeding $\frac{\epsilon}{\sqrt{2}} e^{-LT}\sigma^{-\frac{1}{2}}$ at some point in time is at most $d$ times the probability of the absolute value of an individual component of $Z^t$ exceeding $\frac{\epsilon}{d\sqrt{2}} e^{-LT}\sigma^{-\frac{1}{2}}$ at some point in time. By \eqref{eq:lem1_bound}, we conclude
    \begin{equation}
        \P_x\left( \sup_{0\leq t\leq T} |X_\sigma^t - \phi^t x| \geq \epsilon\right) \leq 2d \,e^{-\gamma\sigma^{-1}}.
    \end{equation}
    This proves the exponential convergence~\eqref{eq:P_exp_convergence}.
\end{proof}

A direct implication of Lemma \ref{lem:exp_close_prob} is exponential convergence in probability of the escape times. 
\begin{lemma}\label{lem:escape_time_convergence}
    Let $X_\sigma^0=x\in M$ be an initial condition that leaves $M$ non-tangentially, and let $\epsilon>0$. There are constants $C>0$ and $\gamma>0$ such that
    \begin{equation}\label{eq:escape_time_convergence}
        \P_x\big(|\tau_\sigma^M - E_0^M(x)| \geq \epsilon\big) \leq C e^{-\gamma \sigma^{-1}}.
    \end{equation}
\end{lemma}
\begin{proof}
    Let $x \in M$ which leaves $M$ non-tangentially. There is a sufficiently small value $\delta \in (0, \epsilon)$ for which $\phi^{E_0^M(x) + \delta} x \notin M$. There is $\epsilon'>0$ such that $\text{dist}(\phi^{E_0^M(x) + \delta}x, M)> \epsilon'$ and $\text{dist}(\phi^t x, \partial M)> \epsilon'$ for all $t\in [0, E_0^M(x) - \delta]$. Now, if $\sup_{0\leq t \leq E_0^M(x) + \delta}  |X_\sigma^t  - \phi^t x| < \epsilon'$, then $\tau_\sigma^M \in [E_0^M (x) - \delta, E_0^M(x) + \delta]$. We apply Lemma \ref{lem:exp_close_prob} to the time-interval $T=E_0^M(x)+\delta$ with the threshold $\epsilon'$. We conclude that there are $C>0$ and $\gamma>0$ such that
    \begin{equation}\label{eq:escape_time_convergence_proof}
        \P_x\big(|\tau_\sigma^M - E_0^M(x)| \geq \delta\big) \leq C e^{-\gamma \sigma^{-1}}.
    \end{equation}
    Since $\delta < \epsilon$, the statement follows.
\end{proof}

The next lemma is the key implication of Assumption \ref{ass:no_stable} and Assumption \ref{ass:uniformly elliptic}. It states that the expected escape time cannot grow exponentially as $\sigma$ goes to zero. The statement is similar to \cite[Chapter 6, Theorem 5.3]{freidlin1998random}, which characterizes the exponential growth rates of the expected escape times without Assumption \ref{ass:no_stable}.
\begin{lemma}\label{lem:uniform_in_x}
    Assume the vector field $v$ satisfies Assumption \ref{ass:no_stable} and that $G$ satisfies the uniform ellipticity Assumption \ref{ass:uniformly elliptic}. For each $\gamma>0$ there are $\sigma^* >0$ and $C>0$ such that for all $\sigma \in (0, \sigma^*)$ and all $x\in M$, we have
    \begin{equation}
        E_\sigma^M(x) \leq C e^{\gamma\sigma^{-1}}.
    \end{equation}
\end{lemma}
\begin{proof}
    The proof of this lemma utilizes the so-called \emph{action-functional} of the family $\{X_\sigma^t\}_{\sigma > 0}$ of random processes. The theory of action functionals is described in detail in \cite[Chapter 3]{freidlin1998random}. To this end, given a weakly differentiable path $\psi:[0,T] \to \R^d$, define the functional
    \begin{equation}
        S_{0,T}(\psi) \coloneqq \frac{1}{4} \int_0^T \left|G(\psi(t))^{-1}\big(v(\psi(t)) - \partial_t \psi(t) \big)\right|^2 \, \dd t.
    \end{equation}
    This functional measures how much ``work'' is needed to make a trajectory follow the path $\psi$, if without forcing it would simply follow the deterministic flow~$\phi^t$. By Assumption \ref{ass:uniformly elliptic}, the matrix $G$ is uniformly elliptic such that there is a constant $\Lambda$ with
    \begin{equation}\label{eq:AF_upper_bound}
         S_{0,T}(\psi) \leq  \frac{1}{4}\Lambda \int_0^T \left|v(\psi(t)) - \partial_t \psi(t) \right|^2 \, \dd t.
    \end{equation}
    
    The functional $S_{0,T}$ is the \emph{action-functional} of the family $\{X_\sigma^t\}_{\sigma > 0}$ of stochastic processes (cf.~\cite[Chapter 4, Theorem 5.1]{freidlin1998random} and \cite[Chapter 3, Example 3.2]{freidlin1998random}). This means that for each $\epsilon>0$ and $\gamma>0$, there is a $\sigma^*>0$ such that for all $\sigma\in (0, \sigma^*)$ we have the bound
    \begin{equation}\label{eq:action_func_upper_bound}
        \P_x\left(\sup_{0\leq t \leq T} |X_\sigma^t - \psi(t) | < \epsilon \right) \geq e^{- (S_{0,T}(\psi) + \frac{1}{2}\gamma)\sigma^{-1}},
    \end{equation}
    where $\psi:[0,T] \to \R^d$ is a weakly differentiable path with $\psi(0) = X_\sigma^0 = x$.\\

    Let $\gamma>0$ be given. Our goal is to find $\epsilon>0$ and for each $x\in M$ a weakly differentiable path $\psi_x:[0,T] \to \R^d$ with the following three properties:
    \begin{enumerate}[(i)]
        \item $\psi_x(0) = x$;
        \item $\sup_{0\leq t\leq T} \text{dist}(\psi_x(t), M) > \epsilon$;
        \item $S_{0,T}(\psi_x) \leq \frac{1}{2} \gamma$.
    \end{enumerate}
    The first two properties state that $\psi_x$ is a path that starts in $x$ and leaves the set $M$. The third property requires that this transition is done at a low cost, measured in terms of the functional~$S_{0,T}(\psi_x)$. We construct these paths $\psi_x$ in four steps. First, we move from $x$ to a nearby point $y\in M$ that exits $M$ under the deterministic flow $\phi^t$. Then we follow the deterministic flow until we reach the boundary~$\partial M$. Then, we move away from $M$ such that (ii) is satisfied. Lastly, we let $\psi_x$ follow the deterministic flow until time $T$, as this does not increase its functional~$S_{0,T}(\psi_x)$.
    
    Let $\alpha >0$ be a small parameter. Due to Assumption \ref{ass:no_stable}, the set of points $x\in M$ with $E_0^M(x)<\infty$ is dense. Since $\overline{M}$ is compact, there is a finite set of points $x_i\in M$ with $E_0^M(x_i)<\infty$ such that the union of the open balls of size $\alpha$ around each of the $x_i$'s covers $\overline{M}$. Set
    \begin{equation}
        K_\alpha \coloneqq \max_i E_0^M(x_i).
    \end{equation}
    
    In addition, for $\theta\in \partial M$ consider the function
    \begin{equation}
        \delta_\alpha(\theta) \coloneqq \sup_{z \in B_\alpha(\theta)} \text{dist}(z, M).
    \end{equation}
    We assumed that the boundary of $M$ equals the boundary of the closure $\overline{M}$. This implies that $\delta_\alpha(\theta) >0$ for all $\theta \in \partial M$. Naturally, the function $\delta_\alpha$ is continuous in $\theta \in \partial M$. Since the boundary $\partial M$ is a compact set, the function $\delta_\alpha$ is bounded from below. We choose $\epsilon_\alpha>0$ small enough such that $\delta_\alpha(\theta) > \epsilon_\alpha$ for all $\theta \in \partial M$.

    The definitions of $K_\alpha$ and $\epsilon_\alpha$ have the following purpose: For each $x\in M$, there is a point $y \in M$ within a distance of $\alpha$ such that $E_0^M(y) \leq K_\alpha$. For each $\theta \in \partial M$, there is a point $z\notin M$ within a distance of $\alpha$ such that $\text{dist}(z, M) >\epsilon_\alpha$. We make use of these properties to construct the paths $\psi_x$.
    
    Set $T_\alpha \coloneqq K_\alpha + 2 \alpha$. Given $x\in M$, construct the path $\psi_x:[0,T_\alpha] \to \R^d$ as the concatenation of four segments. We start in $x$ and go in a straight line to a point $y\in M$ within a distance of $\alpha$ which satisfies $E_0^M(y) \leq K_\alpha$. Then, we follow the deterministic flow until we reach the boundary of $M$ in the point $\theta(y) = \phi^{E_0^M(y)}y \in \partial M$. From $\theta(y)$ go in a straight line to a point $z$ within a distance of $\alpha$ that satisfies $\text{dist}(z,M) > \epsilon_\alpha$. Lastly, follow the deterministic flow until the time $T_\alpha$ is up. These segments are defined as follows:
    \begin{align*}
        \psi_x^1(t) &= x + \frac{t}{\alpha} (y-x), && t\in [0, \alpha],\\
        \psi_x^2(t) &= \phi^t y, && t\in [0, E_0^M(y)],\\
        \psi_x^3(t) &= \theta(y) + \frac{t}{\alpha}(z - \theta(y)), && t\in [0, \alpha], \\
        \psi_x^4(t) &= \phi^t z, && t\in [0, T_\alpha - E_0^M(y) - 2 \alpha].
    \end{align*}
    One can verify that these paths concatenate in the sense that the endpoint of one path is the starting point of the next. Since each path is differentiable, their concatenation is weakly differentiable. Lastly, the total time of the four paths together is $T_\alpha$ such that $\psi_x \coloneqq \psi_x^4 \circ \psi_x^3 \circ\psi_x^2 \circ\psi_x^1:[0,T_\alpha] \to \R^d$ is a weakly differentiable path. The path $\psi_x$ starts in $x$ such that it satisfies property (i). Since $\psi_x$ runs through the point $z \notin M$ which, by construction, has a distance to $M$ of at least $\epsilon_\alpha$, the path satisfies property (ii). Lastly, the functional $S_{0,T}(\psi_x)$ is the sum of the functionals of the four individual segments. Since segments two and four follow the deterministic flow, their functional is zero. By \eqref{eq:AF_upper_bound}, the functionals of the first and third segments are upper bounded by
    \begin{align}
        S_{0,T}(\psi_x^i) \leq \frac{1}{4}\Lambda \int_0^\alpha |v(\psi_x^i(t)) - \partial_t \psi_x^i(t)|^2 \, \dd t.
    \end{align}
    Let $L>0$ be the uniform bound on the strength of the velocity field $v$. By construction, the absolute values of the derivatives of $\psi_x^1$ and $\psi_x^3$ are bounded by $1$. We conclude
    \begin{equation}
        S_{0,T_\alpha}(\psi_x^i) \leq \frac{1}{4} \Lambda\alpha (L+1)^2.
    \end{equation}
    This shows that $S_{0,T_\alpha}(\psi_x) \leq \alpha \Lambda(L+1)^2$. In particular, we can choose $\alpha >0$ small enough such that $S_{0,T_\alpha}(\psi_x) \leq \frac{1}{2} \gamma$. By \eqref{eq:action_func_upper_bound}, we find for $\sigma\in (0, \sigma^*)$
    \begin{equation}
        \P_x\left(\sup_{0\leq t \leq T_\alpha} |X_\sigma^t - \psi_x(t) | < \epsilon_\alpha \right) \geq e^{-\gamma\sigma^{-1} }.
    \end{equation}
    Due to property (ii), the path $\psi_x$ is farther than $\epsilon_\alpha$ from $M$ at some point $t\in [0,T]$. Hence, we conclude
    \begin{equation}
        \P_x\left(\tau_\sigma^M \leq T_\alpha \right) \geq e^{-\gamma\sigma^{-1} }.
    \end{equation}
    This bound holds uniformly over all initial conditions $X_\sigma^0 =x\in M$. Using the Markov property and this uniform bound, we find
    \begin{nalign}[eq:induction]
        \P_x(\tau_\sigma^M > nT_\alpha) &= \E_x\left[\1_{\{\tau_\sigma^M > n T_\alpha\}}\right] \\
        &= \E_x\left[\1_{\{\tau_\sigma^M > n T_\alpha\}} \mid \tau_\sigma^M > (n-1) T_\alpha\right] \P_x \big(\tau_\sigma^M > (n-1) T_\alpha\big) \\
        &= \E_x\left[\P_{X_\sigma^{(n-1)T_\alpha}}(\tau_\sigma^M > T_\alpha)\mid \tau_\sigma^M > (n-1) T_\alpha\right] \P_x \big(\tau_\sigma^M > (n-1) T_\alpha\big)\\
        &\leq (1- e^{-\gamma\sigma^{-1}})\, \P_x \big(\tau_\sigma^M > (n-1) T_\alpha\big).
    \end{nalign}
    By induction, we obtain 
    \begin{equation}
        \P_x(\tau_\sigma^M > nT_\alpha) \leq \big(1-e^{-\gamma\sigma^{-1} }\big)^n .
    \end{equation}
    
    Finally, we conclude that for all $x\in M$ and $\sigma \in (0,\sigma^*)$,
    \begin{align}
        E_\sigma^M (x) &= \int_0^\infty \P_x\big(\tau_\sigma^M > t\big) \, \dd t \\ &\leq \sum_{n=0}^\infty \P_x\big(\tau_\sigma^M > nT_\alpha\big) T_\alpha \\&\leq T_\alpha \sum_{n=0}^\infty\big(1- e^{-\gamma\sigma^{-1} }\big)^n \\&= T_\alpha e^{\gamma\sigma^{-1} }.
    \end{align}
    The constants $T_\alpha$, $\gamma$, and $\sigma^*$ are chosen independently of $x\in M$. This completes the proof.
\end{proof}

\subsection{Proof of Theorem \ref{thm:stoch_inertia}}\label{sec:proof_A}
Let $x^* \in M$ with $E_0^M(x^*)< \infty$ such that $x^*$ leaves $M$ non-tangentially and such that $\partial M$ is twice differentiable in a neighbourhood of $\theta(x^*)$. We first restrict our attention to a subset $U\subset M$ on which $E_0^M$ is finite. \\

\textbf{Construction of the subset $U\subset M$.}\\
Let $I\subset \partial M$ be an open (as a subset of $\partial M$) neighbourhood of $\theta(x^*)$ such that $I$ is twice differentiable. Let $T> E_0^M(x^*)$ and define the set
\begin{equation}\label{eq:def_U}
    U:= \{\phi^{-t} x \mid x \in I, \, t\in(0, T) \}.
\end{equation}
Since $v$ is $C^2$, for $I$ small enough, there are $\delta>0$ and $\epsilon>0$ such that $\text{dist}(\phi^\delta I, M) > \epsilon$. Additionally, $T>E_0^M(x^*)$ and $I\subset \partial M$ can be chosen small enough such that 
$U\subset M$ and $\partial U \cap \partial M = \overline{I}$. By construction, $U$ is an open neighborhood of the trajectory $\{\phi^t x^* \}_{t\in [0, E_0^M(x^*))}$. For $X_\sigma^0=x\in U$, let $\tau_\sigma^U$ be the random escape time from $U$, and let $E_\sigma^U(x)$ be its expectation.

The set $U \subset M$ has the following properties:
\begin{enumerate}[(i)]
    \item For all $x\in U$, we have $E_0^U(x) = E_0^M(x) < T$;
    \item For all $x \in U$, we have $\phi^tx \in U$ for $t\in [0, E_0^M(x))$ and $\text{dist}(\phi^{E_0^M(x) + \delta}x, U) > \epsilon$.
\end{enumerate}

\textbf{Computation of the differential quotients.}\\
Our goal is to compute the derivative of $\sigma \mapsto E_\sigma^M(x^*)$ in $\sigma= 0$. We first compute the derivative of $\sigma \mapsto E_\sigma^U(x^*)$ and then relate the two quantities. The differential quotients are related by
\begin{equation}\label{eq:start}
    \sigma^{-1} \big(E_\sigma^M(x) - E_0^M(x) \big) = \sigma^{-1}\big(E_\sigma^U(x) - E_0^U(x)\big) + \sigma^{-1}\big(E_\sigma^M(x) - E_\sigma^U(x)\big).
\end{equation}
Since $v$ is assumed to be $C^2$ and $I$ is twice differentiable, the implicit function theorem implies that $E_0^M$ is twice differentiable on the closure $\overline U$. By Lemma \ref{lem:uniformly_integrable}, the exit time $\tau_\sigma^U$ has finite mean for small $\sigma>0$. Hence, we can apply Dynkin's formula to the function $E_0^M$ and the stopping time $\tau_\sigma^U$,
\begin{align}
    \E_x[E_0^M(X^{\tau_\sigma^U}_\sigma)] &= E_0^M(x) + \E_x\left[ \int_0^{\tau_\sigma^U} \sigma \, \Gamma(X_\sigma^t) + v(X_\sigma^t)\cdot \nabla E_0^M(X^t_\sigma)\, \dd t \right] \\
    &= E_0^U(x) + \sigma \, \E_x\left[ \int_0^{\tau_\sigma^U} \Gamma(X_\sigma^t) \, \dd t \right] - \E_x[\tau_\sigma^U],
\end{align}
where we used the fact that $v\cdot \nabla E_0^M = -1$. Note that it is sufficient for $E_0^M$ to be defined and twice differentiable on the closure $\overline U$ because due to the stopping time $\tau^U_\sigma$, we never evaluate $E_0^M$ outside of $\overline U$. By definition $\E_x[\tau_\sigma^U] = E_\sigma^U(x)$. After rearranging the equation, we obtain
\begin{equation}\label{eq:EU_diff}
    \sigma^{-1} (E_\sigma^U(x) - E_0^U(x)) = \E_x\left[ \int_0^{\tau_\sigma^U} \Gamma(X_\sigma^t)\, \dd t \right] - \sigma^{-1} \E_x[E_0^M(X_\sigma^{\tau_\sigma^U})].
\end{equation}
Combining \eqref{eq:start} and \eqref{eq:EU_diff} yields
\begin{equation}\label{eq:differential_quotioent_full}
    \sigma^{-1} (E_\sigma^M(x) - E_0^M(x) ) = \E_x\left[ \int_0^{\tau_\sigma^U} \Gamma(X_\sigma^t) \, \dd t \right] - \sigma^{-1} \E_x[E_0^M(X_\sigma^{\tau_\sigma^U})]+ \sigma^{-1}(E_\sigma^M(x) - E_\sigma^U(x)).
\end{equation}
Since $U \subset M$, we naturally have $E_\sigma^M(x) \geq E_\sigma^U(x)$ for all $\sigma \geq 0$, such that
\begin{equation}\label{eq:start_leq}
    \sigma^{-1} (E_\sigma^M(x) - E_0^M(x) ) \geq \E_x\left[ \int_0^{\tau_\sigma^U} \Gamma(X_\sigma^t) \, \dd t \right] - \sigma^{-1} \E_x[E_0^M(X_\sigma^{\tau_\sigma^U})].
\end{equation}

If $\liminf_{\sigma \to 0}\sigma^{-1} (E_\sigma^M(x) - E_0^M(x) ) >0$, then there is stochastic inertia in $x$. Theorem \ref{thm:stoch_inertia} is proven once we show the following limits for $X_\sigma^0 = x \in U$ (in particular for $X_\sigma^0 = x^*$).
\begin{align}
     \text{(L1)}& & &\lim_{\sigma \to 0}  \E_x\left[ \int_0^{\tau_\sigma^U} \Gamma(X_\sigma^t)\, \dd t \right] = \int_0^{E_0^M(x)} \Gamma(\phi^t x)\, \dd t, \\
    \text{(L2)}& & &\lim_{\sigma \to 0} \sigma^{-1} \E_x[E_0^M(X^{\tau_\sigma^U}_\sigma)] = 0.
\end{align}
The rest of this section contains the proof of these two limits.

\begin{lemma}\label{lem:uniformly_integrable}
    There is a $\sigma^*>0$ such that the family of stopping times $\{\tau_\sigma^U\}_{\sigma \in (0, \sigma^*)}$ is uniformly integrable, that is
    \begin{equation}
        \lim_{K \to \infty} \sup_{\sigma \in (0, \sigma^*)} \E[|\tau_\sigma^U| \, \1_{\tau_\sigma^U > K}] = 0.
    \end{equation}
    In particular, $\E[\tau_\sigma^U]<\infty$ for all $\sigma \in (0, \sigma^*)$.
\end{lemma}
\begin{proof}
    Recall the constants $T$, $\delta$ and $\epsilon$ that we used in the definition of the set $U$. Define $T_\delta \coloneqq T + \delta$.
    By the property (ii) of the set $U$, we find
    \begin{equation} 
        \P_x\big( \tau_\sigma^U \leq T_\delta\big) \geq \P_x\left( \sup_{0\leq t \leq T_\delta} |X_\sigma^t - \phi^t x| < \epsilon \right).
    \end{equation}
    The probability on the right-hand side is lower bounded by \eqref{eq:P_exp_convergence} in Lemma \ref{lem:exp_close_prob}, which implies that there are $\gamma>0$ and $C>0$ such that
    \begin{equation}\label{eq:tau_U_upper_bound}
        \P_x\big( \tau_\sigma^U \leq T_\delta\big) \geq 1 - Ce^{-\gamma \sigma^{-1}}.
    \end{equation}
    This bound holds uniformly in $x\in U$. Using the Markov property, analogously to \eqref{eq:induction}, we bound the distribution of $\tau_\sigma^U$ as
    \begin{equation}\label{eq:tau_U_Markov}
        \P_x\big( \tau_\sigma^U > nT_\delta\big) \leq \big( Ce^{-\gamma\sigma^{-1}}\big)^n.
    \end{equation} 
    Let $\sigma^*>0$ be small enough such that $\beta \coloneqq Ce^{-\gamma(\sigma^*)^{-1}}<1$. W.l.o.g.~let $L=mT_\delta$. We compute 
    \begin{align}
        \E_x[\tau_\sigma^U \, \1_{\tau_\sigma^U>L}] &= \P_x\big( \tau_\sigma^U > L\big) \, L + \int_L^\infty \P_x\big(\tau_\sigma^U > t \big) \, \dd t\\
        &\leq \beta^m L +\sum_{n=0}^\infty \P_x\big( \tau_\sigma^U >(m+n)T_\delta\big)\, T_\delta \\
        &\leq \beta^m L + \sum_{n=0}^\infty \beta^{m+n}\, T_\delta\\
        &= \beta^m\left(L+\frac{1}{1-\beta}T_\delta\right).
    \end{align}
    Taking the limit $L\to \infty$ corresponds to taking the limit $m\to \infty$. Then, the right-hand side goes to $0$, which proves that the family $\{\tau_\sigma^U\}_{\sigma\in(0,\sigma^*)}$ is uniformly integrable.
\end{proof}
Using this lemma, we prove the first of the two limits.

\begin{proposition}[L1]
    $\lim_{\sigma \to 0}  \E_x\left[ \int_0^{\tau_\sigma^U} \Gamma(X_\sigma^t)\, \dd t \right] = \int_0^{E_0^M(x)} \Gamma(\phi^t x)\, \dd t.$
\end{proposition}
\begin{proof}
    We first show that the random integral converges in probability and then use the Vitali convergence theorem to conclude that the convergence also holds in expectation.

    The functional $\Gamma$ is continuous on $U$. Since $U$ is bounded, $\Gamma$ is bounded by some constant $K>0$ and uniformly continuous. Let $\epsilon >0$. By uniform continuity, there is a value $\delta>0$ such that
    \begin{equation}
        |x-y|<\delta \quad  \Rightarrow \quad |\Gamma(x) - \Gamma(y)| < \epsilon, \qquad \forall x,y\in U.
    \end{equation}
    On the event $\sup_{0\leq t\leq E_0^M(x)}|X_\sigma^t - \phi^t x|< \delta$, and $|\tau_\sigma^U - E_0^M(x)| < \epsilon$, we bound the difference
    \begin{align}
        \left|\int_0^{\tau_\sigma^U} \Gamma(X_\sigma^t)\, \dd t - \int_0^{E_0^M(x)} \Gamma(\phi^t x)\, \dd t\right| &\leq K |\tau_\sigma^U - E_0^M(x)| + \left|\int_0^{E_0^M(x)} \Gamma(X_\sigma^t) -\Gamma( \phi^t x) \, \dd t \right| \\
        &\leq  (K + E_0^M(x)) \epsilon.
    \end{align}

    By Lemma \ref{lem:exp_close_prob} the term $\sup_{0\leq t\leq E_0^M(x)}|X_\sigma^t - \phi^t x|$ converges to zero in probability. By Lemma \ref{lem:escape_time_convergence}, the term $|\tau_\sigma^U - E_0^U(x)|$ converges to zero in probability. We conclude that for any $\epsilon>0$
    \begin{equation}
        \lim_{\sigma\to 0} \P_x \left( \left|\int_0^{\tau_\sigma^U} \Gamma(X_\sigma^t)\, \dd t - \int_0^{E_0^M(x)} \Gamma(\phi^t x)\, \dd t\right| \geq (K+E_0^M(x))\epsilon\right) = 0.
    \end{equation}
    Since $\epsilon$ can be chosen arbitrarily small, this shows that the random integral converges in probability. 
    
    To go from convergence in probability to convergence in expectation, we make use of the Vitali convergence theorem \cite[Theorem 4.6.3]{durrett2019probability}. It states that a uniformly integrable family of random variables that converges in probability also converges in expectation. Hence, it suffices to prove that the random integral is uniformly integrable. By construction of $K$, we have
    \begin{equation}
        \left|\int_0^{\tau_\sigma^U} \Gamma(X_\sigma^t)\, \dd t \right| \leq K |\tau_\sigma^U|.
    \end{equation}
    Lemma \ref{lem:uniformly_integrable} proved that the family $\{\tau_\sigma^U\}_{\sigma \in (0, \sigma^*)}$ is uniformly integrable. As a consequence, also the family of the random integrals is uniformly integrable. This shows that the Vitali convergence theorem is applicable and we conclude
    \begin{equation}
        \lim_{\sigma \to 0}  \E_x\left[ \int_0^{\tau_\sigma^U} \Gamma(X_\sigma^t)\, \dd t \right] = \int_0^{E_0^M(x)} \Gamma(\phi^t x)\, \dd t.
    \end{equation}
\end{proof}

\begin{proposition}[L2]
    $\lim_{\sigma \to 0} \sigma^{-1} \E_x\left[E_0^M(X_\sigma^{\tau_\sigma^U})\right] = 0.$
\end{proposition}
Note that $E_0^M$ is zero on the boundary portion $I\subset \partial M$. Thus, the limit (L2) has the interpretation that exiting $U$ happens with overwhelming probability through~$I$.
\begin{proof}
    By definition of the stopping time, we find $X_\sigma^{\tau_\sigma^U} \in \partial U$. If $X_\sigma^{\tau_\sigma^U} \in I \subset \partial M$, then $E_0^M(X_\sigma^{\tau_\sigma^U}) = 0$. Furthermore, we have $E_0^M(x) \leq T$ for all $x\in \overline{U}$, where $T$ is the constant in the definition of $U$, cf.~\eqref{eq:def_U}. We conclude
    \begin{equation}
        \E_x\left[E_0^M(X_\sigma^{\tau_\sigma^U})\right] \leq \P_x\left(X_\sigma^{\tau_\sigma^U} \notin I\right) T.
    \end{equation}
    Assume $X_\sigma^0 = x$. Recall the constants $\delta>0$ and $\epsilon>0$ from property (ii) of $U$, and let $T_\delta:= T+ \delta$. W.l.o.g.~we can choose $\epsilon$ small enough such that $\sup_{0\leq t \leq T_\delta} |X_\sigma^t - \phi^t x| < \epsilon$ implies $\tau_\sigma^U \leq T_\delta$ and $X_\sigma^{\tau_\sigma^U} \in I$. We conclude
    \begin{equation}
        \P_x(X_\sigma^{\tau_\sigma^U} \notin I) \leq \P_x\left( \sup_{0\leq t \leq T_\delta} |X_\sigma^t - \phi^t x| \geq \epsilon \right).
    \end{equation}
    By \eqref{eq:P_exp_convergence} in Lemma \ref{lem:exp_close_prob} there are $C>0$ and $\gamma>0$ such that
    \begin{equation}\label{eq:exit_I_prob}
        \P_x(X_\sigma^{\tau_\sigma^U} \notin I) \leq  \P_x\left( \sup_{0\leq t \leq T_\delta} |X_\sigma^t - \phi^t x| \geq \epsilon \right) < Ce^{-\gamma \sigma^{-1}}.
    \end{equation}
    We conclude
    \begin{equation}
        \lim_{\sigma \to 0} \sigma^{-1} \E_x[E_0^M(X_\sigma^{\tau_\sigma^U})] \leq \lim_{\sigma \to 0} \sigma^{-1} \P_x(X_\sigma^{\tau_\sigma^U} \notin I) \, T \leq \lim_{\sigma\to 0} \sigma^{-1} C e^{-\gamma \sigma^{-1}}T = 0.
    \end{equation}
    Since the limit cannot be negative, it must be $0$ and (L2) is proven.
\end{proof}

Together, (L1) and (L2) imply the statement of Theorem~\ref{thm:stoch_inertia}.

\subsection{Proof of Theorem \ref{thm:derivative}}
Recall \eqref{eq:differential_quotioent_full}, the representation of the differential quotient of $E_\sigma^M(x^*)$ in $\sigma = 0$:
\begin{equation}
    \sigma^{-1} (E_\sigma^M(x) - E_0^M(x) ) = \E_x\left[ \int_0^{\tau_\sigma^U} \Gamma(X_\sigma^t) \, \dd t \right] - \sigma^{-1} \E_x[E_0^M(X_\sigma^{\tau_\sigma^U})]+ \sigma^{-1}(E_\sigma^M(x) - E_\sigma^U(x)).
\end{equation}
The Propositions (L1) and (L2) show that the limits of the first two summands exist as $\sigma\to 0$. Inserting the respective limits yields
\begin{equation}
    \lim_{\sigma \to 0}  \sigma^{-1} (E_\sigma^M(x) - E_0^M(x) ) = \int_0^{E_0^M(x)} \Gamma(\phi^t x)\, \dd t + \lim_{\sigma \to 0} \sigma^{-1}(E_\sigma^M(x) - E_\sigma^U(x)).
\end{equation}

Now, Theorem \ref{thm:derivative} follows from the following proposition:

\begin{proposition}[L3]
    Assume that $v$ satisfies Assumption \ref{ass:no_stable} and that $G$ satisfies the uniform ellipticity Assumption \ref{ass:uniformly elliptic}. Then, $\lim_{\sigma \to 0} \sigma^{-1}(E_\sigma^M(x) - E_\sigma^U(x)) = 0.$
\end{proposition}
\begin{proof}
    The process $(X_\sigma^t)_{t\geq 0}$ has the strong Markov property. Hence, the expected escape time from $M$ is given by
    \begin{equation}
        E_\sigma^M(x) = E_\sigma^U(x) + \E_x \left[ \E_{X_\sigma^{\tau_\sigma^U}} \left[ \tau_\sigma^M \right] \right].
    \end{equation}
    If $X_\sigma^{\tau_\sigma^U} \in I \subset \partial M$, then the term $\E_{X_\sigma^{\tau_\sigma^U}} \left[ \tau_\sigma^M \right]$ is zero. Hence,
    \begin{equation}
        E_\sigma^M(x) - E_\sigma^U(x) \leq \P_x(X_\sigma^{\tau_\sigma^U}\notin I) \left( \sup_{y\in \partial U} E_\sigma^M(y)\right).
    \end{equation}
    In the previous proof (cf.~\eqref{eq:exit_I_prob}), we showed that $\P_x(X_\sigma^{\tau_\sigma^U}\notin I) < Ce^{-\gamma \sigma^{-1}}$ for some constants $C>0$ and $\gamma>0$. Set $\gamma^* = \frac{\gamma}{2}$. Invoking now Assumptions~\ref{ass:no_stable} and~\ref{ass:uniformly elliptic}, by Lemma \ref{lem:uniform_in_x} there are constants $C^*>0$ and $\sigma^*>0$  that do not depend on $x\in M$ such that for all $\sigma\in (0, \sigma^*)$ we find
    \begin{equation}
        E_\sigma^M(x) \leq C^* e^{\gamma^* \sigma^{-1}}.
    \end{equation}
    We conclude
    \begin{align}
        \lim_{\sigma \to 0} \sigma^{-1}(E_\sigma^M(x) - E_\sigma^U(x)) &\leq 
        \lim_{\sigma \to 0} \sigma^{-1} \left(\P_x(X_\sigma^{\tau_\sigma^U}\notin I) \left( \sup_{y\in \partial U} E_\sigma^M(y)\right)\right)\\
        &\leq \lim_{\sigma \to 0} \sigma^{-1} \left(C e^{-\gamma \sigma^{-1}} \, C^* e^{\gamma^* \sigma^{-1}} \right)\\&= 0.
    \end{align}
    Since the limit cannot be negative, it must be zero. This proves (L3).
\end{proof}

\section{Outlook}
\label{sec:outlook}

\subsection{Stochastic inertia as a regime-average quantity}
\label{ssec:regime_average}

As discussed in the Introduction, a motivation for this work originated from observations in theoretical meteorology where the change of regime lifetimes under addition of noise was studied \cite{kwasniok2014enhanced, dorrington2023interaction}. 
Such a regime is typically understood as a neighbourhood of some saddle equilibrium or, more generally, a subset $M \subset \mathbb R^d$ in phase space.
However, for a continuous-time model of the form~\eqref{eq:SDE}, considering regime persistence via the exit time distribution for $\sigma>0$ upon entry at some $x\in \partial M$ is not possible in a straightforward manner: due to properties of Brownian motion, almost every path of~\eqref{eq:SDE} would leave $M$ again instantaneously (see also \cite{schoeller2026noise} for a discussion of this issue).
In other words, the analysis of the impact of noise on exit times has to be supplemented by a suitable description of regime entry.

For such an analysis, one may not define a regime simply through just one set $M$ but also by an ``entry'' set $M_{\textnormal{in}}$ such that $\textnormal{cl}(M_{\textnormal{in}})\subset \textnormal{int}(M)$. A regime's duration is then measured as the time between entering $M_{\textnormal{in}}$ and exiting $M$. 

The second challenge is to derive a suitable notion of regime average.
We can try to compute the exit times $E_{\sigma}$ of a regime $M$ with respect to the stationary distribution $\mu_{\sigma}$ of the SDE~\eqref{eq:SDE}, conditional on some set of entry points~$\Gamma_{\textnormal{in}} \subset M_{\textnormal{in}}$, leading to $\int_{\Gamma_{\textnormal{in}}} E_{\sigma}\,\dd \mu_{\sigma}|_{\Gamma_{\textnormal{in}}}$. Assuming sufficient smoothness in $\sigma$, the infinitesimal change of these exit times is 
\begin{equation}\label{eq:derivative_Gamma_in}
\left( \int_{\Gamma_{\textnormal{in}}} E_{\sigma}\,\dd\mu_{\sigma}|_{\Gamma_{\textnormal{in}}} \right)' = \int_{\Gamma_{\textnormal{in}}} E_{\sigma}'\,\dd\mu_{\sigma}|_{\Gamma_{\textnormal{in}}} + \int_{\Gamma_{\textnormal{in}}} E_{\sigma}\,\dd(\mu_{\sigma}|_{\Gamma_{\textnormal{in}}})',
\end{equation}
where $'$ denotes the derivative with respect to~$\sigma$. This leads to a question of linear response theory with respect to noise for~$\mu_\sigma$~\cite{lucarini2012stochastic,galatolo2020quadratic}, whose restriction to $\Gamma_{\textnormal{in}}$  can also be seen as a \emph{reinjection probability}, see e.g.~\cite{del2012effect, krause2014noise, elaskar2011reinjection}.

If such a computation is feasible, one may make the set $M_{\textnormal{in}}$ depend on a small parameter~$\epsilon$, e.g.,
\begin{equation*}
    M_{\textnormal{in}}(\epsilon) := \{x\in \R^d \mid B_\epsilon(x) \subset M\}.
\end{equation*}
 For fixed $\epsilon$, we can then study the dependence of the average lifetime on $\sigma$, in particular the derivative of the average lifetime in $\sigma = 0$. This derivative still depends on $\epsilon$. Taking $\epsilon \to 0$, one could conclude the analysis by showing that the derivative converges to the quantity in~\eqref{eq:derivative_Gamma_in}.

\subsection{Stochastic limit cycles}
\label{ssec:limit_cycle}

A setting closely related to the one discussed in this article concerns the existence of an attracting, hyperbolic limit cycle $\Gamma \subset \mathbb R^d$ for the ODE dynamics of~\eqref{eq:SDE}, i.e.~when $\sigma =0$.
In the deterministic situation, such a limit cycle with period $\tau_\Gamma$ can be well described via the \emph{isochrons} $W(x)$ which are the leaves of the stable manifold $W^{\textnormal{s}} (\Gamma) = \cup_{x \in \Gamma} W(x)$, identifying the submanifolds of the absorbing neighbourhood whose asymptotic phase converges to $x \in \Gamma$ \cite{Guckenheimer75}. The isochrons can also be characterized as the level sets of an \emph{isochron map} $\theta: W^{\textnormal{s}} (\Gamma) \to \Gamma$.

Now, under the impact of stochasticity, i.e., in the situation of stochastically perturbed limit cycles, there have been several approaches to defining \emph{stochastic isochronicity}: using eigenfunctions of the associated infinitesimal generator $\mathcal L$ \cite{LindnerThomas}, identifying random isochrons as the random stable leaves of a random limit cycle as a stochastically covariant object (in a quenched sense) \cite{EngelKuehn}, or introducing them as sections with the mean first return time to the same section being a certain constant \cite{Caoetal2020, SchwabedalPikovsky}.
In terms of metastability around such limit cycles, a framework of random Poincar\'e maps has been developed \cite{BaudelBerglund}.

In this context, a question arises that is intimately related to our study and its techniques: how does the asymptotic average oscillation speed behave under variation of arbitrarily small noise? To answer this question, we need a way to quantify the average oscillation speed of a stochastic limit cycle. A natural choice is to use the isochron map $\theta$. 

Let $\phi^t$ denote the deterministic flow associated with~\eqref{eq:SDE} for $\sigma =0$, $\Gamma$ be the hyperbolic attracting limit cycle with period $\tau_\Gamma$
and $\omega:= \tau_\Gamma^{-1}$. We may consider the parameterization $\alpha: \Gamma \to S^1$ with
\begin{equation}
    \alpha(\phi^t x) = \alpha(x) + \omega t \mod 1, 
\end{equation} 
which can be extended to $W^{\textnormal{s}} (\Gamma)$ via $\alpha(x) = \alpha (\theta(x))$. For the solution $X_\sigma^t$ of the SDE~\eqref{eq:SDE}, with $\sigma \geq 0$ and $X_\sigma^0 \in W^{\textnormal{s}}(\Gamma)$, we consider the continuous map
    $a: [0,\infty) \to S^1, \, a(t) \coloneqq \alpha(X_{\sigma}^t)$
which is well-defined as long as $X_\sigma^t \in W^{\textnormal{s}}(\Gamma)$. There is a unique  lift $\Tilde{a}$ with $\Tilde{a}(0) = a(0) \in [0,1)$ and $a(t) = \Tilde{a}(t) \mod 1$, counting the number of rotations. Assuming the limit is well-defined, the average oscillation speed of the stochastic trajectory is given by
\begin{equation}\label{eq:winding_number}
    \omega_\sigma \coloneqq \lim_{t\to \infty} \frac{1}{t} \Tilde{a}(t).
\end{equation}
It turns out that for uniformly elliptic noise, this limit does not converge in general, see \cite{spitzer1958some} and \cite{vakeroudis2015windings} for a more recent presentation. The issue is that there is a point in the center of the limit cycle in which the isochron map is discontinuous. A stochastic trajectory that gets close to such a point can wind around the point arbitrarily fast in either direction. The variance of such excursions is too large such that they dominate the winding number $\omega_\sigma$ and prevent convergence of \eqref{eq:winding_number}. 

In \cite{giacomin2018small} the authors circumvent this obstruction by considering the dynamics only in the neighborhood of the limit cycle, where there are no singularities. For a stable limit cycle $\Gamma$, the expected time to leave such a neighborhood is exponential in $\sigma^{-1}$ such that their analysis is valid for long times, as long as $\sigma$ is small. Applying It\^o's formula to the phase parametrization $\alpha:W^s(\Gamma) \to \R$ and using $\nabla \alpha^\top v = \omega$ yields
\begin{equation}
    \E_x[\Tilde{a}(t)] = \Tilde{a}(0)+ \omega t + \sigma \E_x\left[ \int_0^t \operatorname{trace}\left( G(X_\sigma^s)^\top [\operatorname{Hess}\alpha(X_\sigma^s)] \,G(X_\sigma^s)\right) \dd s\right].
\end{equation}
A formal derivation along the lines of Section \ref{sec:proof_A} suggests that the derivative of $\omega_\sigma$ in $\sigma=0$ is given by
    \begin{equation}
        \partial_\sigma \omega_\sigma \big|_{\sigma=0} = \int_\Gamma \operatorname{trace}\left( G(x)^\top [\operatorname{Hess}\alpha(x)] \,G(x)\right) \, \mu_0(\dd x),
    \end{equation}
    where $\mu_0$ is the occupation measure of the limit cycle $\Gamma$. This formula is consistent with the results derived in \cite{giacomin2018small}. The problem with making this formula rigorous without the restriction to a neighborhood of $\Gamma$ is that the Hessian of $\alpha$ is unbounded around singularities.

    An alternative approach to quantify the oscillation speed of a stochastic limit cycle is to ask: How long does it take on average for the next full rotation to be completed? We use the same phase parametrization as before and define
    \begin{equation}
        \tau_\sigma \coloneqq \inf \{t\geq 0 \mid \Tilde{a}(t) - \Tilde{a}(0) \geq 1\}.
    \end{equation}
    We denote $E_\sigma(x) = \E_x[\tau_\sigma]$. The average time until the next full rotation is given by
    \begin{equation}
        \int_{\R^d} E_\sigma(x) \, \mu_\sigma (\dd x),
    \end{equation}
    where $\mu_\sigma$ is the stationary distribution of the SDE with noise strength $\sigma$. Like in \eqref{eq:derivative_Gamma_in}, the $\sigma$-derivative of the average time until the next full rotation is given by
    \begin{equation}
        \left(\int_{\R^d} E_\sigma \, \mu_\sigma \right)' = \int_{\R^d} E_\sigma' \, \mu_\sigma + \int_{\R^d} E_\sigma \mu_\sigma'.
    \end{equation}
    The stopping time $\tau_\sigma$ is closely related to the escape time $\tau_\sigma^M$ which we have studied in this article, and the derivative $\partial_\sigma E_\sigma\big|_{\sigma=0}$ can be computed analogously to Section \ref{sec:proofs}. Hence, the only missing part to compute the derivative in $\sigma=0$ of the average next rotation time is the linear response $\partial_\sigma \mu_\sigma |_{\sigma=0}$. This response can be computed using the methods from \cite{galatolo2020quadratic}, but making the result rigorous likely requires careful treatment of arising technicalities.

\section*{Acknowledgements}
The authors thank Zachary Adams, Christian Kuehn, and Joshua Dorrington for their insightful feedback and discussions on the topics of this article.

This research has been supported by Deutsche Forschungsgemeinschaft (DFG) through grant CRC 1114 Scaling Cascades in Complex Systems, Project Number 235221301, Project A08, “Characterization and Prediction of Quasi-Stationary Atmospheric States”.
ME was additionally supported by the Dutch Research Council (NWO), in the framework of the program VI.Vidi.233.133 `A Rigorous Framework for Transient Random Dynamics'.

\section*{Statement on the use of artificial intelligence}
We acknowledge the use of AI tools (LLMs), in particular ChatGPT 6 Astra, to generate code for Section~\ref{ssec:relaxation_oscillator}, to aid with the technical formulas of Appendix \ref{app:derivation}, and Claude Opus 5.5 to proofread the paper. The code has been verified by the authors. Appendix~\ref{app:derivation} has been carefully checked and rewritten by the authors, who are solely responsible for the correctness of the results.

\bibliographystyle{alpha}
\bibliography{bibliography}

\newcommand{\etalchar}[1]{$^{#1}$}
\begin{thebibliography}{NKMS90}

\bibitem[AE10]{altmann2010noise}
Eduardo~G. Altmann and Antonio Endler.
\newblock Noise-enhanced trapping in chaotic scattering.
\newblock {\em Phys. Rev. Lett.}, 105(24):244102, Dec 2010.

\bibitem[BB17]{BaudelBerglund}
Manon Baudel and Nils Berglund.
\newblock Spectral theory for random {P}oincar\'{e} maps.
\newblock {\em SIAM J. Math. Anal.}, 49(6):4319--4375, 2017.

\bibitem[BEGK04]{bovier2004metastability}
Anton Bovier, Michael Eckhoff, V{\'e}ronique Gayrard, and Markus Klein.
\newblock Metastability in reversible diffusion processes {I}: Sharp asymptotics for capacities and exit times.
\newblock {\em Journal of the European Mathematical Society}, 6(4):399--424, 2004.

\bibitem[Ber13]{berglund2013kramers}
Nils Berglund.
\newblock {Kramers}' law: Validity, derivations and generalisations.
\newblock {\em Markov Processes and Related Fields}, 19(3):459--490, 2013.

\bibitem[BG04]{berglund2004noise}
Nils Berglund and Barbara Gentz.
\newblock On the noise-induced passage through an unstable periodic orbit {I}: Two-level model.
\newblock {\em Journal of Statistical Physics}, 114(5--6):1577--1618, 2004.

\bibitem[BG06]{BerglundGentz2006}
Nils Berglund and Barbara Gentz.
\newblock {\em Noise-induced phenomena in slow-fast dynamical systems}.
\newblock Probability and its Applications (New York). Springer-Verlag London, Ltd., London, 2006.
\newblock A sample-paths approach.

\bibitem[BG14]{berglund2014noise}
Nils Berglund and Barbara Gentz.
\newblock On the noise-induced passage through an unstable periodic orbit {II}: General case.
\newblock {\em SIAM Journal on Mathematical Analysis}, 46(1):310--352, 2014.

\bibitem[BGK12]{BerglundGentzKuehn}
Nils Berglund, Barbara Gentz, and Christian Kuehn.
\newblock Hunting {F}rench ducks in a noisy environment.
\newblock {\em J. Differential Equations}, 252(9):4786--4841, 2012.

\bibitem[BR16]{bouchet2016generalisation}
Freddy Bouchet and Julien Reygner.
\newblock Generalisation of the {E}yring--{K}ramers transition rate formula to irreversible diffusion processes.
\newblock {\em Annales Henri Poincar{\'e}}, 17(12):3499--3532, 2016.

\bibitem[CLT20]{Caoetal2020}
Alexander Cao, Benjamin Lindner, and Peter~J. Thomas.
\newblock A partial differential equation for the mean-return-time phase of planar stochastic oscillators.
\newblock {\em SIAM J. Appl. Math.}, 80(1):422--447, 2020.

\bibitem[DGW92]{dayan1992stochastic}
Ido Dayan, Moshe Gitterman, and George~H Weiss.
\newblock Stochastic resonance in transient dynamics.
\newblock {\em Physical Review A}, 46(2):757--761, 1992.

\bibitem[DP23]{dorrington2023interaction}
Joshua Dorrington and Tim Palmer.
\newblock On the interaction of stochastic forcing and regime dynamics.
\newblock {\em Nonlinear Processes in Geophysics}, 30(1):49--62, 2023.

\bibitem[dRSE12]{del2012effect}
Ezequiel del Rio, Miguel~AF Sanju{\'a}n, and Sergio Elaskar.
\newblock Effect of noise on the reinjection probability density in intermittency.
\newblock {\em Communications in Nonlinear Science and Numerical Simulation}, 17(9):3587--3596, 2012.

\bibitem[Dur19]{durrett2019probability}
Rick Durrett.
\newblock {\em Probability: theory and examples}, volume~49 of {\em Cambridge Series in Statistical and Probabilistic Mathematics}.
\newblock Cambridge University Press, fifth edition, 2019.

\bibitem[EDRD11]{elaskar2011reinjection}
Sergio Elaskar, Ezequiel Del~Rio, and Jose~M Donoso.
\newblock Reinjection probability density in type-iii intermittency.
\newblock {\em Physica A: Statistical Mechanics and its Applications}, 390(15):2759--2768, 2011.

\bibitem[EK21]{EngelKuehn}
Maximilian Engel and Christian Kuehn.
\newblock A random dynamical systems perspective on isochronicity for stochastic oscillations.
\newblock {\em Comm. Math. Phys.}, 386(3):1603--1641, 2021.

\bibitem[Eyr35]{eyring1935activated}
Henry Eyring.
\newblock The activated complex in chemical reactions.
\newblock {\em The Journal of Chemical Physics}, 3(2):107--115, 1935.

\bibitem[Fra91]{franaszek1991influence}
Marek Franaszek.
\newblock Influence of noise on the mean lifetime of chaotic transients.
\newblock {\em Phys. Rev. A}, 44(6):4065--4067, Sep 1991.

\bibitem[FW98]{freidlin1998random}
Mark~I. Freidlin and Alexander~D. Wentzell.
\newblock {\em Random Perturbations of Dynamical Systems}, volume 260 of {\em Grundlehren der mathematischen Wissenschaften}.
\newblock Springer, second edition, 1998.

\bibitem[GJT09]{gutkin2009inhibition}
Boris~S Gutkin, J{\"u}rgen Jost, and Henry~C Tuckwell.
\newblock Inhibition of rhythmic neural spiking by noise: the occurrence of a minimum in activity with increasing noise.
\newblock {\em Naturwissenschaften}, 96(9):1091--1097, 2009.

\bibitem[GPS18]{giacomin2018small}
Giambattista Giacomin, Christophe Poquet, and Assaf Shapira.
\newblock Small noise and long time phase diffusion in stochastic limit cycle oscillators.
\newblock {\em Journal of Differential Equations}, 264(2):1019--1049, 2018.

\bibitem[Gra89]{graham1989macroscopic}
Robert Graham.
\newblock Macroscopic potentials, bifurcations and noise in dissipative systems.
\newblock In Frank Moss and P.~V.~E. McClintock, editors, {\em Noise in Nonlinear Dynamical Systems}, volume~1, chapter~7, pages 225--278. Cambridge University Press, 1989.

\bibitem[GS20]{galatolo2020quadratic}
Stefano Galatolo and Julien Sedro.
\newblock Quadratic response of random and deterministic dynamical systems.
\newblock {\em Chaos: An Interdisciplinary Journal of Nonlinear Science}, 30(2):023113, 2020.

\bibitem[Guc75]{Guckenheimer75}
J.~Guckenheimer.
\newblock Isochrons and phaseless sets.
\newblock {\em J. Math. Biol.}, 1(3):259--273, 1975.

\bibitem[HKN04]{helffer2004quantitative}
Bernard Helffer, Markus Klein, and Francis Nier.
\newblock Quantitative analysis of metastability in reversible diffusion processes via a {W}itten complex approach.
\newblock {\em Matem{\'a}tica Contempor{\^a}nea}, 26:41--86, 2004.

\bibitem[KEdR14]{krause2014noise}
Gustavo Krause, Sergio Elaskar, and Ezequiel del Rio.
\newblock Noise effect on statistical properties of type-i intermittency.
\newblock {\em Physica A: Statistical Mechanics and its Applications}, 402:318--329, 2014.

\bibitem[Kra40]{kramers1940brownian}
Hendrik~Anthony Kramers.
\newblock {Brownian} motion in a field of force and the diffusion model of chemical reactions.
\newblock {\em Physica}, 7(4):284--304, 1940.

\bibitem[Kue15]{KuehnBook}
Christian Kuehn.
\newblock {\em Multiple time scale dynamics}, volume 191 of {\em Applied Mathematical Sciences}.
\newblock Springer, Cham, 2015.

\bibitem[Kue17]{kuehn2017quenched}
Christian Kuehn.
\newblock Quenched noise and nonlinear oscillations in bistable multiscale systems.
\newblock {\em Europhysics Letters}, 120(1):10001, 2017.

\bibitem[Kwa14]{kwasniok2014enhanced}
Frank Kwasniok.
\newblock Enhanced regime predictability in atmospheric low-order models due to stochastic forcing.
\newblock {\em Philosophical Transactions of the Royal Society A: Mathematical, Physical and Engineering Sciences}, 372(2018):20130286, 2014.

\bibitem[Lev50]{levinson1950first}
Norman Levinson.
\newblock The first boundary value problem for $\varepsilon {\Delta} u+ {A} (x, y) u_x+ {B} (x, y) u_y+ {C} (x, y) u= {D} (x, y)$ for small $\varepsilon$.
\newblock {\em Annals of Mathematics}, 51(2):428--445, 1950.

\bibitem[LJG{\etalchar{+}}20]{lu2020inverse}
Lulu Lu, Ya~Jia, Mengyan Ge, Ying Xu, and Anbang Li.
\newblock Inverse stochastic resonance in {H}odgkin--{H}uxley neural system driven by {G}aussian and non-{G}aussian colored noises.
\newblock {\em Nonlinear Dynamics}, 100(1):877--889, 2020.

\bibitem[LT11]{lai2011transient}
Ying-Cheng Lai and Tam{\'a}s T{\'e}l.
\newblock {\em Transient chaos: complex dynamics on finite time scales}, volume 173 of {\em Applied Mathematical Sciences}.
\newblock Springer, 2011.

\bibitem[Luc12]{lucarini2012stochastic}
Valerio Lucarini.
\newblock Stochastic perturbations to dynamical systems: a response theory approach.
\newblock {\em Journal of Statistical Physics}, 146(4):774--786, 2012.

\bibitem[MS77]{matkowsky1977exit}
Bernard~J Matkowsky and Zeev Schuss.
\newblock The exit problem for randomly perturbed dynamical systems.
\newblock {\em SIAM Journal on Applied Mathematics}, 33(2):365--382, 1977.

\bibitem[MS96]{mantegna1996noise}
Rosario~N Mantegna and Bernardo Spagnolo.
\newblock Noise enhanced stability in an unstable system.
\newblock {\em Physical Review Letters}, 76(4):563--566, 1996.

\bibitem[NKMS90]{naeh1990direct}
T~Naeh, MM~K{\l}osek, BJ~Matkowsky, and Z~Schuss.
\newblock A direct approach to the exit problem.
\newblock {\em SIAM Journal on Applied Mathematics}, 50(2):595--627, 1990.

\bibitem[Rei94]{reimann1994suppression}
Peter Reimann.
\newblock Suppression of deterministic diffusion by noise.
\newblock {\em Physical Review E}, 50(2):727--735, 1994.

\bibitem[RY99]{revuz2013continuous}
Daniel Revuz and Marc Yor.
\newblock {\em Continuous Martingales and {Brownian} Motion}, volume 293 of {\em Grundlehren der mathematischen Wissenschaften}.
\newblock Springer, third edition, 1999.

\bibitem[Sch09]{schuss2009theory}
Zeev Schuss.
\newblock {\em Theory and applications of stochastic processes: an analytical approach}, volume 170.
\newblock Springer Science \& Business Media, 2009.

\bibitem[SCK{\etalchar{+}}26]{schoeller2026noise}
Henry Schoeller, Robin Chemnitz, P{\'e}ter Koltai, Maximilian Engel, and Stephan Pfahl.
\newblock Noise-induced enhancement of regime lifetimes--{A} data-driven approach using deterministic trajectories.
\newblock {\em arXiv preprint arXiv:2604.27991}, 2026.

\bibitem[SP13]{SchwabedalPikovsky}
Justus T.~C. Schwabedal and Arkady Pikovsky.
\newblock Phase description of stochastic oscillations.
\newblock {\em Phys. Rev. Lett.}, 110(20):204102, May 2013.

\bibitem[Spi58]{spitzer1958some}
Frank Spitzer.
\newblock Some theorems concerning 2-dimensional {Brownian} motion.
\newblock {\em Transactions of the American Mathematical Society}, 87(1):187--197, 1958.

\bibitem[TJ12]{tuckwell2012analysis}
Henry~C Tuckwell and J{\"u}rgen Jost.
\newblock Analysis of inverse stochastic resonance and the long-term firing of {H}odgkin--{H}uxley neurons with {G}aussian white noise.
\newblock {\em Physica A: Statistical Mechanics and its Applications}, 391(22):5311--5325, 2012.

\bibitem[TL14]{LindnerThomas}
Peter~J. Thomas and Benjamin Lindner.
\newblock Asymptotic phase for stochastic oscillators.
\newblock {\em Phys. Rev. Lett.}, 113(25):254101, Dec 2014.

\bibitem[Vak15]{vakeroudis2015windings}
Stavros Vakeroudis.
\newblock On the windings of complex-valued {Ornstein}--{Uhlenbeck} processes driven by a {Brownian} motion and by a stable process.
\newblock {\em Stochastics An International Journal of Probability and Stochastic Processes}, 87(5):766--793, 2015.

\bibitem[vdP20]{van1920theory}
B.~van~der Pol.
\newblock A theory of the amplitude of free and forced triode vibrations.
\newblock {\em Radio Review}, 1:701--710, 754--762, 1920.

\bibitem[YJ17]{yamakou2017simple}
Marius~E Yamakou and J{\"u}rgen Jost.
\newblock A simple parameter can switch between different weak-noise--induced phenomena in a simple neuron model.
\newblock {\em Europhysics Letters}, 120(1):18002, 2017.

\end{thebibliography}

\appendix

\section{Derivation of the algorithm in Section \ref{sec:computing}}
\label{app:derivation}

This section justifies the algorithm from Section \ref{sec:computing} which computes the integral
\begin{equation}
    \int_0^{E_0^M(x)} \Gamma(\phi^tx)\, \dd t.
\end{equation}
Let $x\in M$ be a point that leaves $M$ non-tangentially in a point $\theta \in \partial M$, and assume that $\partial M$ is twice differentiable around $\theta$. Let $N:\R^d\to \R$ be a boundary-defining function satisfying the properties (i) and (ii) from Section \ref{sec:computing}. In the following we justify the formulas from Step 1 to Step 4.

The deterministic escape time $E_0^M(x)$ is characterized by the condition
\begin{equation}\label{eq:implicit_diff}
    N\big(\phi^{E_0^M(x)}x\big) = 0.
\end{equation}
The first zero of $N(\phi^t x)$ determines the exit time $T$ and exit point $\theta$ in Step~1. Differentiating \eqref{eq:implicit_diff} twice gives the boundary data in Step~2. We then derive the equations that transport these data along the trajectory in Step~3 and accumulate the integral in Step~4.

To derive the boundary data in Step~2, we work directly at $\theta$. Recall the notation $n= \nabla N(\theta)$, $b=v(\theta)$, and $c=n^\top b$. Since $c>0$, the implicit function theorem gives a twice continuously differentiable local extension $u$ of $E_0^M$ near $\theta$, characterized by
\begin{equation}
    N\big(\phi^{u(y)}y\big)=0,\qquad u(\theta)=0.
\end{equation}
This extension allows negative hitting times for points outside $M$. Step 2 derives a formula for
\begin{equation}
    g(T) = \nabla u(\theta), \qquad H(T) = \operatorname{Hess} u(\theta).
\end{equation}
For a displacement $\epsilon h$, where $h\in\R^d$ is an arbitrary fixed direction, define
\begin{equation}\label{eq:def_T_i}
    T_h(\epsilon)=u(\theta+\epsilon h),\qquad
    \theta_h(\epsilon)=\phi^{T_h(\epsilon)}(\theta+\epsilon h).
\end{equation}
In particular, $T_h(0)=0$ and $\theta_h(0)=\theta$. Primes denote derivatives with respect to $\epsilon$, with $h$ held fixed. Differentiating $N(\theta_h(\epsilon))=0$ and the expression for $\theta_h(\epsilon)$ at $\epsilon=0$ gives
\begin{align}
    0&=n^\top\theta_h'(0),\label{eq:first_derivative}\\
    \theta_h'(0)&=h+bT_h'(0).\label{eq:der_theta}
\end{align}
We insert the second identity into the first and rearrange,
\begin{equation}\label{eq:der_T}
    T_h'(0)=-\frac{n^\top h}{c}.
\end{equation}
Consequently, the corresponding change in the exit point is
\begin{equation}\label{eq:def_z}
    \theta_h'(0)=h-\frac{n^\top h}{c}b=Ph.
\end{equation}
Since $T_h'(0) = g(T)^\top h$, this gives $g(T)=-n/c$ and explains the matrix $P$ in Step~2.

Differentiating the exit point $\theta_h(\epsilon)$ a second time and evaluating at $\epsilon=0$ gives
\begin{nalign}[eq:der2_theta]
    \theta_h''(0)
    &=2T_h'(0)Dv(\theta)h+T_h'(0)^2Dv(\theta)b+bT_h''(0)\\
    &=T_h'(0)Dv(\theta)(Ph+h)+bT_h''(0),
\end{nalign}
where we used $Ph=h+bT_h'(0)$ in the second equality.
On the other hand, differentiating $N(\theta_h(\epsilon))=0$ twice at $\epsilon=0$ gives
\begin{equation}
    0=(Ph)^\top\operatorname{Hess}N(\theta)(Ph)+n^\top\theta_h''(0).
\end{equation}
Inserting \eqref{eq:der2_theta} and using $a=Dv(\theta)^\top n$, we can solve for $T_h''(0)$:
\begin{equation}\label{eq:der2_T}
    h^\top H(T)h=T_h''(0)
    =-\frac{(Ph)^\top\operatorname{Hess}N(\theta)(Ph)}{c}
        +\frac{n^\top h}{c^2}a^\top(Ph+h).
\end{equation}
Substituting $Ph=h-b(n^\top h)/c$ in the last term, and using that $h$ is arbitrary, gives exactly the Hessian in \eqref{eq:boundary_derivatives}. Together with $g(T)=-n/c$, these are the terminal values for the backward equations in Step~3.

To derive the equations in Step~3, write $u=E_0^M$, so that $g(t)=\nabla u(X^t)$ and $H(t)=\operatorname{Hess}u(X^t)$. We use the identity
\begin{equation}
    u(\phi^\epsilon y)=u(y)-\epsilon,
\end{equation}
valid near each interior point of the trajectory for sufficiently small $\epsilon$. Differentiating this identity twice with respect to $y$ gives
\begin{align}
    \nabla u(y)&=D\phi^\epsilon(y)^\top\nabla u(\phi^\epsilon y),\\
    \operatorname{Hess}u(y)&=D\phi^\epsilon(y)^\top\operatorname{Hess}u(\phi^\epsilon y)D\phi^\epsilon(y)\\
        &\quad+\sum_{k=1}^d \partial_k u(\phi^\epsilon y)D^2\phi_k^\epsilon(y).
\end{align}
At $\epsilon=0$, we have $D\phi^0(y)=\Id$ and $D^2\phi^0(y)=0$. Their derivatives with respect to $\epsilon$ are $Dv(y)$ and $D^2v(y)$, respectively. Thus, differentiating the preceding relations in $\epsilon$ at $\epsilon=0$, with $y=\phi^t x=X^t$, gives
\begin{align}
    0&=Dv(X^t)^\top g(t)+\partial_t g(t),\\
    0&=Dv(X^t)^\top H(t)+\partial_t H(t)+H(t)Dv(X^t)\\
        &\quad+\sum_{k=1}^d g_k(t)\operatorname{Hess}v_k(X^t).
\end{align}
Rearranging gives \eqref{eq:derivative_transport}. Since the gradient and Hessian are known at $t=T$ from Step~2, these equations are solved backward along the trajectory from Step~1.

Finally, Step~4 is elementary. This concludes the derivation of the formulas from Step 1 to Step 4.

\end{document}